\documentclass[11pt]{article}
\usepackage{amsmath,amsfonts,amssymb, amsthm,enumerate,graphicx}
\usepackage{tikz,authblk}
\usepackage{subfig}
\usepackage{url}
\usepackage{stackengine,scalerel}
\usetikzlibrary{arrows.meta}

\newtheorem{theorem} {{\textsf{Theorem}}}
\newtheorem{proposition}[theorem]{{\textsf{Proposition}}}
\newtheorem{corollary}[theorem]{{\textsf{Corollary}}}

\newtheorem{definition}[theorem]{{\textsf{Definition}}}
\newtheorem{remark}[theorem]{{\textsf{Remark}}}

\newtheorem{lemma}[theorem]{{\textsf{Lemma}}}

\newcommand{\xdownarrow}[1]{
	{\left\downarrow\vbox to #1{}\right.\kern-\nulldelimiterspace}
}
\begin{document} 
	
\title{Semi-equivelar gems of PL $d$-manifolds}
\author{Biplab Basak and Manisha Binjola$^1$}
	
\date{}
	
\maketitle
	
\vspace{-10mm}
\begin{center}
		
\noindent {\small Department of Mathematics, Indian Institute of Technology Delhi, New Delhi 110016, India.$^2$}

\footnotetext[1]{Corresponding author}
		
\footnotetext[2]{{\em E-mail addresses:} \url{biplab@iitd.ac.in} (B. Basak), \url{binjolamanisha@gmail.com} (M. Binjola).}
		
\medskip
		
\date{February 6, 2024}
\end{center}
	
\hrule
	
\begin{abstract}
We define the notion of $(p_0,p_1,\dots,p_d)$-type semi-equivelar gems for closed connected PL $d$-manifolds, related to the regular embedding of gems $\Gamma$ representing $M$ on a surface $S$ such that the face-cycles at all the vertices of $\Gamma$ on $S$ are of the same type. The term is inspired by semi-equivelar maps of surfaces. Given a surface $S$ having non-negative Euler characteristic, we find all regular embedding types on $S$ and then construct a genus-minimal semi-equivelar gem (if it exists) of each such type embedded on $S$. Moreover, we present constructions of the following semi-equivelar gems:
\begin{enumerate}[$(1)$]
\item For each closed connected surface $S$, we construct a genus-minimal semi-equivelar gem that represents $S$. In particular, for $S=\#_n (\mathbb{S}^1 \times \mathbb{S}^1)$ (resp., $\#_n(\mathbb{RP}^2)$), the semi-equivelar gem of type $((4n+2)^3)$ (resp., $((2n+2)^3)$) is constructed.
			
\item  For a closed connected orientable PL $d$-manifold $M$ (where $d \geq 3$) of regular genus at most $1$, we show that $M$ admits a genus-minimal semi-equivelar gem if and only if $M$ is a lens space.
\end{enumerate}
Moreover, if we consider semi-equivelar gems with $2$-gons then for a closed connected orientable $d$-manifold $M$  (where $d \geq 3$) with $\mathcal{G}(M)\leq 1$, $M$ admits a genus-minimal semi-equivelar gem (with bigons).
\end{abstract}
	
\noindent {\small {\em MSC 2020\,:} Primary 52C20; Secondary 52B70, 05C10, 05C15, 57Q15.
		
}
	
\medskip
	
\section{Introduction}
	
A  map is an embedding of a connected graph on a surface.
For each vertex $u$, there is a cycle of polygonal faces $C_u$ with vertex at $u$. Two face cycles are of the same type if they are the same as an ordered sequence up to the cyclic shift. A map is called {\it semi-equivelar} if face-cycles are of the same type for all vertices. From \cite{bk08}, a map on a surface is called equivelar if there exist numbers $p,q$ such that every vertex is $q$-valent and every facet contains exactly $p$ vertices and edges. In other words, a semi-equivelar map is equivelar if the face cycle consists of similar polygonal faces at each vertex.
	
In \cite{mtu14}, a classification of semi-equivelar maps on surfaces of Euler characteristic $-1$ is given. In \cite{dm17} and \cite{dm18}, the authors discuss semi-equivelar maps on the torus and Klein bottle. In \cite{dm22},  semi-equivelar maps on the $2$-sphere are given.

A graph encoded manifold, in short, `gem', $(\Gamma,\gamma)$ of a connected compact PL-manifold is a certain kind of edge-colored graph that represents the manifold (see Subsection \ref{crystal}). It is known that every $(d+1)$-regular gem can be embedded on a surface regularly (see Subsection \ref{sec:genus}). The motivation behind the term semi-equivelar gems for PL $d$-manifold is taken from the semi-equivelar maps on surfaces.
	
Given a surface $S$ having non-negative Euler characteristic, we find all regular embedding types on $S$ and then construct a semi-equivelar gem (if it exists) of each such type embedded on $S$. Further, for each connected surface $K$, we construct a genus-minimal semi-equivelar gem that represents $K$. In particular, for $K=\#_n (\mathbb{S}^1 \times \mathbb{S}^1)$ (resp., $\#_n(\mathbb{RP}^2)$), the genus-minimal semi-equivelar gem of type $((4n+2)^3)$ (resp., $((2n+2)^3)$) is constructed. We also prove the necessary and sufficient condition for a closed connected orientable PL $d$-manifold $M$ $(d \geq 3)$ with regular genus at most $1$ to admit a genus-minimal semi-equivelar gem is that $M$ must be a lens space. To avoid any ambiguity, we note that if  $\Gamma$ is a semi-equivelar gem for a closed $d$-manifold $M$ then it means that $\Gamma$ represents $M$ and there exists a surface $S$ on which it embeds regularly such that the face-cycles at all the vertices of $\Gamma$ on $S$ are of the same type (see Definition \ref{def:semiequivelar}). We conclude the article by defining the term semi-equivelar gems (with bigons) and showing that for a closed connected orientable $d$-manifold $M$  $(d \geq 3)$ with $\mathcal{G}(M)\leq 1$, $M$ admits a genus-minimal semi-equivelar gem (with bigons).
	
\section{Preliminaries}
The Crystallization theory gives a combinatorial technique to represent any piecewise-linear (PL) manifold using edge-colored graphs.
	
\subsection{Graph encoded manifolds (gem)} \label{crystal}
	
For a multigraph $\Gamma= (V(\Gamma), E(\Gamma))$ without loops, edges are labeled (or colored) by $\Delta_d:=\{0,1, \dots, d\}$. The coloring is called a proper edge-coloring if every two adjacent edges have different colors. The members of the set $\Delta_d$ are called the {\it colors} of $\Gamma$. More precisely, for a proper edge-coloring, there exists a surjective map $\gamma : E(\Gamma) \to \Delta_d$ with  $\gamma(e_1) \ne \gamma(e_2)$ for any two adjacent edges $e_1$ and $e_2$. A  graph with a proper edge coloring is denoted by $(\Gamma,\gamma)$.
If the degree of each vertex in a graph $(\Gamma,\gamma)$ is  $d+1$ then it is said to be {\it $(d+1)$-regular}. We refer to \cite{bm08} for standard terminology on graphs. All spaces will be considered in the PL-category.

A  regular {\it $(d+1)$-colored graph} is a pair $(\Gamma,\gamma)$, where $\Gamma$ is $(d+1)$-regular and $\gamma$ is a proper edge-coloring. If there is no confusion with coloration,  $\Gamma$ can be used instead of $(\Gamma,\gamma)$ for $(d+1)$-colored graphs. For each $\mathcal{C} \subseteq \Delta_d$ with a cardinality of $q$, the graph $\Gamma_\mathcal{C} =(V(\Gamma), \gamma^{-1}(\mathcal{C}))$ is a $q$-colored graph with edge-coloring $\gamma|_{\gamma^{-1}(\mathcal{C})}$. For $\{j_1,j_2,\dots,j_q\} \subset \Delta_d$, $g(\Gamma_{\{j_1,j_2, \dots, j_q\}})$ or $g_{j_1j_2 \dots j_q}$ denotes the number of connected components of the graph $\Gamma_{\{j_1, j_2, \dots, j_q\}}$.  A graph $(\Gamma,\gamma)$ is called {\it contracted} if  subgraph $\Gamma_{\hat{i}}:=\Gamma_{\Delta_d\setminus i}$ is connected for all $i \in \Delta_d$

Let $\mathbb{G}_d$ denote the set of regular $(d+1)$-colored graphs. For each  $(\Gamma,\gamma) \in \mathbb{G}_d$, a corresponding $d$-dimensional simplicial cell-complex ${\mathcal K}(\Gamma)$ is constructed as follows:
	
\begin{itemize}
\item{} for each vertex $u\in V(\Gamma)$, take a $d$-simplex $\sigma(u)$ with vertices labeled by $\Delta_d$;
		
\item{} corresponding to each edge of color $j$ between $u,v\in V(\Gamma)$, identify the ($d-1$)-faces of $\sigma(u)$ and $\sigma(v)$ opposite to $j$-labeled vertices such that the same labeled vertices coincide.
\end{itemize}
If the geometric carrier $|\mathcal{K}(\Gamma)|$  is (PL) homeomorphic to a PL $d$-manifold $M$ then $\mathcal{K}(\Gamma)$ is said to be a pseudotriangulation of $M$, and $(\Gamma,\gamma)$ is said to be a gem (graph encoded manifold) of $M$ (or is said to represent $M$). It is a known fact that every PL $d$-manifold admits a gem. A $(d+1)$-colored gem of a closed connected (PL) $d$-manifold is called a {\em crystallization} if it is contracted. From the above construction, it can be visualized that $|\mathcal{K}(\Gamma)|$ is a closed connected PL $d$-manifold if and only if for each $c\in \Delta_d$, $\Gamma_{\hat{c}}$ represents $\mathbb{S}^{d-1}$.
	
Let $(\Gamma,\gamma)$ and $(\bar{\Gamma},\bar{\gamma})$   be two $(d+1)$-colored graphs with color set $\Delta_d$ and $\bar{\Delta}_d$ respectively. Then $I(\Gamma):=(I_V,I_c):\Gamma \to \bar{\Gamma}$ is called an {\em isomorphism} if $I_V: V(\Gamma) \to V(\bar{\Gamma})$ and $I_c:\Delta_d \to \bar{\Delta}_d$ are bijective maps such that $uv$ is an edge of color $i \in \Delta_d$ if and only if $I_V(u)I_V(v)$ is an edge of color $I_c(i) \in \bar{\Delta}_d$. The graphs $(\Gamma, \gamma)$ and $(\bar{\Gamma}, \bar{\gamma})$ are then said to be isomorphic.
	
\subsection{Regular Genus of closed PL $d$-manifolds}\label{sec:genus}

From \cite{g81}, it is known that if $(\Gamma,\gamma)\in \mathbb{G}_d$ is a bipartite (resp. non bipartite) $(d+1)$-regular colored graph which represents a closed connected PL $d$-manifold $M$ then for each cyclic permutation $\varepsilon=(\varepsilon_0,\dots,\varepsilon_d)$ of $\Delta_d$, there exists a regular embedding of $\Gamma$ into an orientable (resp. non orientable) surface $S$. A regular embedding is an embedding where each region is alternately bounded by a bi-colored cycle with colors $\varepsilon_i,\varepsilon_{i+1}$ for some $i$ (addition is modulo $d + 1$). Moreover, the Euler characteristic $\chi_\varepsilon(\Gamma)$ of the orientable (resp. non orientable) surface  $S$ satisfies
$$\chi_\varepsilon(\Gamma)=\sum_{i \in \mathbb{Z}_{d+1}}g_{\varepsilon_i\varepsilon_{i+1}} + (1-d)\frac{V(\Gamma)}{2},$$ 
and the genus (resp. half of genus) $\rho_ \varepsilon$ of $S$ satisfies
$$\rho_ \varepsilon(\Gamma)=1-\frac{\chi_\varepsilon(\Gamma)}{2}.$$
The regular genus $\rho(\Gamma)$ of $(\Gamma,\gamma)$ is defined as
$$\rho(\Gamma)= \min \{\rho_{\varepsilon}(\Gamma) \ | \  \varepsilon \ \mbox{ is a cyclic permutation of } \ \Delta_d\}.$$
The regular genus of $M$ is defined as 
$$\mathcal G(M) = \min \{\rho(\Gamma) \ | \  (\Gamma,\gamma)\in \mathbb{G}_d \mbox{ represents } M\}.$$

A  map is an embedding of a connected graph on a surface, where faces are polygons. For a vertex $u$ of a map $X$ on a surface, the faces containing $u$ form a cycle $C_u=P_0, P_1, \dots, P_k, P_0$ in the dual graph of $X.$ If for any two vertices $u$ and $v$, the cycles of faces $C_u,C_v$ are of the same type then the map $X$ is called semi-equivelar. For $p_0,p_1,\dots, p_k\geq 3$, if $C_u=P_0, P_1, \dots, P_k, P_0$ where $P_j$ is a $p_i$-gon then the semi-equivelar map $X$ is denoted by $(p_0, p_1, \dots, p_k)$.  If $n_i$ number of adjacent faces are $p_i$ polygons, then $X$ can be written as $(p_0^{n_0},p_1^{n_1},\dots,p_m^{n_m})$, $n_0, \dots , n_m \geq 1$. 
	
\begin{proposition}[\cite{dm17}]\label{S=0}
Let $X$ be a semi-equivelar map on a surface $S$. If $\chi(S) = 0$ then $X$ is of type $(3^6)$, $(3^4, 6^1)$, $(3^3, 4^2)$, $(3^2, 4^1, 3^1, 4^1)$, $(4^4)$, $(3^1, 6^1, 3^1, 6^1)$, $(3^2, 6^2)$, $(3^2, 4^1, 12^1)$, $(3^1, 4^1, 3^1, 12^1)$, \, $(3^1, 4^1, 6^1, 4^1)$, \, $(3^1, 4^2, 6^1)$,  \, $(6^3)$, \, $(3^1, 12^2)$, \, $(4^1, 8^2)$, \, $(5^2, 10^1)$, $(3^1, 7^1, 42^1)$, $(3^1, 8^1, 24^1)$, $(3^1, 9^1, 18^1)$, $(3^1, 10^1, 15^1)$, $(4^1, 5^1, 20^1)$ or $(4^1, 6^1, 12^1).$
\end{proposition}
	
\section{Main results}
	
For a $(d+1)$-colored gem $\Gamma$, it is known to have the existence of a surface $S$ in which $\Gamma$ can be embedded regularly.  For each vertex $x$ of $\Gamma$, there is a cycle of faces $P_0, P_1, \dots, P_d$, where each $P_i$ is a polygon of length at least $4$ and its edges are colored by two alternate colors.      
\begin{definition}\label{def:semiequivelar}
{\rm
Let $\Gamma$ be a gem that represents a closed PL $d$-manifold $M$. Let $\Gamma$ embed regularly on a surface $S$. If the face-cycles  $C=P_0, P_1, \dots, P_d$ at all the vertices in the embedding of $\Gamma$ on the surface $S$ are of the same type then $\Gamma$ is called {\it semi-equivelar}  gem of $M$, and we say that $M$ admits a semi-equivelar gem. 
			
If $P_i$ is  $p_i$-polygon for $i \in \{ 0,1,\dots,d\}$ then $\Gamma$ is called  { \it $(p_0,p_1,\dots,p_d)$-type semi-equivelar} gem. If there are $n_i$ number of adjacent $p_i$-polygons then $\Gamma$ is called {\it $(p_0^{n_0},p_1^{n_1},\dots,p_m^{n_m})$-type semi-equivelar} gem.}
\end{definition}
In the article, we address two primary problems. Initially, we explore all regular embedding types (if any) on a surface with non-negative Euler characteristics, seeking instances where a semi-equivelar gem of the specified type exists. Subsequently, we inquire whether it is feasible to get a semi-equivelar gem for a given closed manifold $M$.

\begin{proposition}\label{prop:dattasir}
If $\Gamma$ is a semi-equivelar gem embedded regularly on a surface $S$ with $\chi(S)=0$ then $\Gamma$ is of type $(4^4)$, $(6^3)$, $(8^2,4^1)$ or $(4,6,12)$.
\end{proposition}

\begin{proof}
Let  $\Gamma$ be a semi-equivelar gem. Since a regular embedding implies the even length of polygonal faces, Proposition \ref{S=0}  gives the required result.
\end{proof}
	
\begin{lemma}\label{lemma:possibletypes}
Let $\Gamma$ be a semi-equivelar $(d+1)$-colored gem of order $p$ which embeds regularly on a surface $S$ such that $\chi(S)\geq 1$. Then the possible embedding types of $\Gamma$ on $S$ are: $(4^3)$, $ (4,6,8)$, $(4,6,10)$, $(6^2,4^1)$,  $(4^2,q)$, $q \geq 4$ is  even. Also, the values of $p$ are as follows:
$$\begin{matrix}
Type & (4^3)& (4,6,8) &(4,6,10)& (6^2,4^1)&  (4^2,q)\\
\mbox{ if }\chi(S)=1&4&24&60 &12&q\\
\mbox{ if }\chi(S)=2&8 &48& 120 &24&2q
\end{matrix}$$
\end{lemma}

\begin{proof}
Let $\Gamma$ be  $({p_0},{p_1},\dots,{p_d})$-type semi-equivelar gem embedded regularly on $S$. Since embedding is regular, $p_i$ is even. Let $p$ and $d^\prime$ be the number of vertices and degree of each vertex of $\Gamma$ respectively. Let $q_0,q_1, \dots, q_l$ be the polygons of different sizes, where $q_j=p_i$. Let $k_i$ be the number of $q_i$-polygons. Now, $\sum_{i=0}^l{k_i}=d+1=d^\prime$. Let $V, E$, and $F$ be the number of vertices, edges, and faces of the embedding of $\Gamma$ respectively. Then $V=p$, $E=p d^\prime/2$ and $F=p(\frac{k_0}{q_0}+\frac{k_1}{q_1}+ \dots +\frac{k_l}{q_l})$. Thus, we have
\begin{align}
\Big(1-\frac{d^\prime}{2}+\frac{k_0}{q_0}+\frac{k_1}{q_1}+ \dots +\frac{k_l}{q_l}\Big)=\frac{\chi(S)}{p}. \label{1}
\end{align}
By assumption, $q_i \geq 4$. This gives $k_i/q_i \leq k_i/4$ which further implies
\begin{equation} \label{2}
d^\prime \leq 4-\frac{4 \chi(S)}{p}. 
\end{equation}
		
\noindent \textbf {Case 1.} Let $\chi(S)=1$. Equation \eqref{2} gives $d^\prime \leq 4-4/p$.  Since $p \geq 4$, $d^\prime \leq 3$. This implies  $d^\prime=3$. Therefore, $(k_0,k_1,\dots,k_j)=(3),(2,1)$ or $(1,1,1)$.
Also, Equation \eqref{1} reduces to 
\begin{align}\label{3}
\frac{k_0}{q_0}+\frac{k_1}{q_1}+ \dots +\frac{k_l}{q_l}=\frac{1}{p}+\frac{1}{2}.
\end{align}
Consider $(k_0)=(3)$. We get $q_0=4$. Otherwise, Since $q_0 \geq 6$, $3/q_0 \leq 1/2$. Then  Equation \eqref{1} implies $p \leq 0$, which is not possible. Therefore, $(q_0^{k_0})=(4^3)$ and $p=4$ from Equation \eqref{1}.
		
Now, consider $(k_0,k_1)=(2,1)$. Let us first consider $q_0 > q_1$. Equation \eqref{3} gives $1/p \leq 3/q_1-1/2$ which further gives $q_1=4$. Using Equation  \eqref{3} again, we get $q_0 \leq 6$. Since $q_0 \neq q_1$, $q_0=6$. Now, substituting $q_0$ in Equation \eqref{3}, we get $p=12$. Therefore, $(q_0^{k_0},q_1^{k_1})=(6^2,4^1)$. If we consider $q_1 > q_0$ then on the similar notes, we get  $(q_0^{k_0},q_1^{k_1})=(4^2,p^1)$.

Finally, we consider $(k_0,k_1,k_2)=(1,1,1)$. Without loss of generality, we assume $q_0<q_1<q_2.$ Now, Equation \eqref{3} excludes the values $q_0 \geq 6$, $q_1 \geq 8$ and $q_2 \geq 12$ on the similar lines as in the last paragraph. Therefore, $(q_0^{k_0},q_1^{k_1},q_2^{k_2})=(4,6,8)$ or $(4,6,10)$ and number of vertices of  $\Gamma$ is $24$ or $60$ respectively.
		
\smallskip
		
\noindent \textbf{Case 2.} Let $\chi(S)=2$. By the same arguments as above, we get $(i)$ $(q_0^{k_0},p)=(4^3,8)$, for $k_0=3$, $(ii)$ $(q_0^{k_0},q_1^{k_1},p)=(6^2,4,24)$ or $(4^2,p/2,p)$ for $(k_0,k_1)=(2,1)$ and $(iii)$ $(q_0^{k_0},q_1^{k_1},q_2^{k_2},p)=(4,6,8,48)$ or $(4,6,10,120)$ for $(k_0,k_1,k_2)=(1,1,1)$. Thus, the result follows.
\end{proof}
	
\begin{lemma} \label{newlemma:4^4}
The $(4^4)$-type semi-equivelar gems which can be embedded regularly on $\mathbb{S}^1 \times \mathbb{S}^1$ represent the $3$-sphere $\mathbb{S}^3$ or a lens space $L(p,q)$ for some  co-prime integers $p>q\geq 1$.
\end{lemma}
	
\begin{proof}
We first find the possible $(4^4)$-type semi-equivelar gems $\Gamma$ that embed regularly on $\mathbb{S}^1 \times \mathbb{S}^1$. We assume that the $01$-, $12$-, $23$-, and $03$-colored cycles are of length $4$. Let there be $k$ number of $02$-colored cycles. Let $u_1, u_2, \dots, u_{2m}$ be the vertices of a $02$-colored cycle, for some $m$. Since each component of $\Gamma_{\{0,1,2\}}$ represents a $2$-sphere, $\Gamma_{\{0,1,2\}}$ is bipartite. If possible, let $u_1u_{2t}$ be an edge of color $1$, for some $t\leq m$. Since $01$- and $12$-colored cycles are of length $4$, $u_2u_{2t-1}, u_3u_{2t-2}, \dots, u_tu_{t+1}$ are edges of color $1$. Then $\Gamma_{\{0,1,2\}}$ has a $01$- or $12$-colored cycle of length $2$ with vertices $u_t$ and $u_{t+1}$. This is a contradiction. Therefore, no two vertices of the same $02$-colored cycle are joined by a $1$-colored edge. By similar arguments, we can say that no two vertices of the same $02$-colored cycle are joined by a $3$-colored edge. Let $v_1, v_2, \dots, v_{2m'}$ be the vertices of another $02$-colored cycle, for some $m'$, such that $u_1v_1$ is an edge of color $1$. Since $01$- and $12$-cycles are of length $4$, $m=m'$, and $u_iv_i$ is also an edge of color $1$, for $1\leq i \leq 2m$. A similar argument is true for an edge of color $3$ as well. Therefore, the length of each $02$-cycle is the same, and $k$ is even.

Let $2s$ be the length of $02$-cycles. We enumerate these cycles and their vertices. Let $C_l,$ $1 \leq l \leq k$, be $02$-colored cycles, and $\{u_{i}^l \mid 1 \leq i \leq 2s\}$ be the vertex-set of the $C_l$ cycle. Assume that the edge $u^l_{i}u^l_{i+1}$ is colored 0 for odd $i$ and colored 2 for even $i$, considering $u^l_{2s+1}=u^l_1$. Without loss of generality, up to isomorphism, we assume that $u_1^1u_1^2$ is an edge labeled by color $1$. Since $01$- and $12$-cycles are of length $4$, $u^1_ju^2_j$ is an edge labeled by color $1$, $\forall$ $j \in \{1, \dots, 2s\}$. Similarly, since $23$- and $03$-cycles are of length $4$, $u^2_ju^3_j$ is an edge labeled by color $3$, $\forall$ $j \in \{1, \dots, 2s\}$. Continuing in this way, we get that $u^l_ju^{l+1}_j$ is an edge labeled by color $1$ (resp. color $3$) for $l$ odd (resp. $l$ even), $1 \leq j \leq 2s$, $1 \leq l < k$. Therefore, each vertex in $C_1$ and $C_k$ is adjacent to a $c$-colored edge, where $c \neq 3$. Since $\Gamma$ embeds regularly on $\mathbb{S}^1 \times \mathbb{S}^1$, $\Gamma$ is bipartite. Therefore, $u_1^k$ is joined to $u_t^1$ by a $3$-colored edge, where $t$ is odd and $1 \leq t \leq 2s$. If $u_1^k$ is joined to $u_{1+2r}^1$ by a $3$-colored edge, where $0\leq r \leq s-1$, then the fact that `$23$- and $03$-colored cycles are of length $4$' implies that $u_j^k$ is joined to $u^1_{j+2r}$ (addition in the subscript is modulo $2s$) by a $3$-colored edge, where $1\leq j \leq s$. Thus, we have at most $s$ number of $(4^4)$-type semi-equivelar gems $\Gamma$ up to isomorphism (see Figure \ref{fig:L(p,q)}).

Now, let us find the $3$-manifolds $|\mathcal{K}(\Gamma)|$ corresponding to these $(4^4)$-type semi-equivelar gems $\Gamma$. Let $C_{ij}$ denote the subgraph of $\Gamma$ consisting of  $02$-colored cycles $C_i$, $C_j$, and the edges joining the vertices of these cycles. For an example, $C_{12}$ denotes  $02$-colored cycles $C_1$, $C_2$ and the $1$-colored edges $u^1_ju^2_j$ for each $j$,  $C_{23}$ denotes $C_2$, $C_3$ and the $3$-colored edges $u^2_ju^3_j$ for each $j$, and so on. Each component $\Gamma_{\hat{3}}$ of $\Gamma$ represents subcomplex of a 3-ball in $\mathcal{K}(\Gamma)$. In other words, $C_{12},C_{34},\dots, C_{k-1k}$ represent subcomplexes of 3-balls in $\mathcal{K}(\Gamma)$. For each even $l$, $ 1 \leq l \leq k-1$, the $3$-colored edges between $C_l$ and $C_{l+1}$ identify the half portions of boundaries of two $3$-balls (represented by $C_{l-1\, l}$ and $C_{l+1\, l+2}$) and thus giving a $3$-ball. Therefore, the union of $C_{12},C_{23},\dots ,C_{l\,l+1} , \dots, C_{k-1\,k}$ denoted by $\mathcal{B}$ represents a $3$-ball. Now, adding the $3$-colored edges  $u^k_ju^1_{j+2r}$ to $\mathcal{B}$ forms $\Gamma$. If $u^1_j$ is joined to $u^k_{j}$, then $\Gamma$  represents $3$-sphere $\mathbb{S}^3$.  If $u^k_j$ is joined to $u^1_{j+2r}$ (addition in the subscript is of modulo $2s$) for some $1\leq r < s$, then $\Gamma$ represents lens space $L(p,q)$, for some co-prime integers $p>q\geq 1$, where $\frac{s}{p}=\frac{r}{q}$. To see this, consider $\frac{s}{p}=\frac{r}{q}=d$, where $gcd(s,r)=d$. Let $P$ be a 3-dimensional polytope (in $\mathbb{R}^3$) whose vertices are $u^{\pm}=(0,0, \pm 1)$, $v_i= (\cos \frac{2 i \pi}{p}, \sin\frac{2 i \pi}{p},0)$, $0\leq i\leq p-1$, and the boundary of $P$ consists of $2p$ triangles $\alpha^{\pm}_i :=$ convex hull of $\{u^{\pm}, v_i, v_{i+1}\}$ (here summations in the subscripts are modulo $p$). Introducing $2d-1$ vertices along each edge $v_i v_{i+1}$ and connecting them with $u^{\pm}$ yields a new 3-dimensional polytope $P'$ with a boundary isomorphic to $\partial(|\mathcal{K}(\mathcal{B})|)$. Therefore, there is a homeomorphism from $|\mathcal{K}(\mathcal{B})|$ to the polytope $P$, which sends $2d$ triangles of $\partial (|\mathcal{K}(\mathcal{B})|)$ to a single triangle of $\partial(P)$, and adding the $3$-colored edges  $u^k_ju^1_{j+2r}$ to $\mathcal{B}$ identifies $\alpha^{+}_i$ with $\alpha^{-}_{i+q}$ in $P$, where $u^{+}, v_i, v_{i+1}$ get identified with $u^{-}, v_{i+q}, v_{i+1+q}$, respectively. Therefore, the gem $\Gamma$ represents $L(p, q)$. This completes the proof.
\end{proof}

 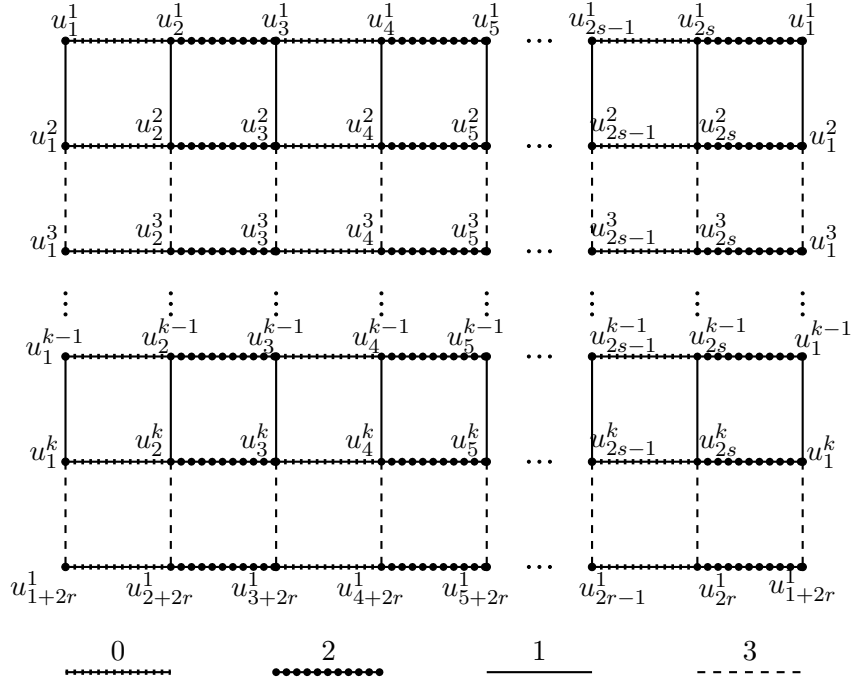
\begin{figure}[ht]
\tikzstyle{ver}=[]
\tikzstyle{vertex}=[circle, draw, fill=black!100, inner sep=0pt, minimum width=4pt]
\tikzstyle{edge} = [draw,thick,-]
\centering
\begin{tikzpicture}[scale=1.4]

%%%%%%%%%%%%%%%%%%%%%%%%%%%%%%%%%%%%%%%%%%%%%%%%%%%%%%%%%%%%%%%%%%%%%%%%%%%%%%%%%%%%%%%%%%%%%%%
\begin{scope}[shift={(1,1)}]

\draw [line width=2pt, line cap=rectengle, dash pattern=on 1pt off 2] (1,1) -- (2,1);
\path[edge] (1,1) -- (2,1);

\draw [line width=3pt, line cap=round, dash pattern=on 0pt off 1.3\pgflinewidth] (2,1) -- (3,1);
\path[edge] (2,1) -- (3,1);

\draw [line width=2pt, line cap=rectengle, dash pattern=on 1pt off 2] (3,1) -- (4,1);
\path[edge] (3,1) -- (4,1);

\draw [line width=3pt, line cap=round, dash pattern=on 0pt off 1.3\pgflinewidth] (4,1) -- (5,1);
\path[edge] (4,1) -- (5,1);

\filldraw (5.4,1) circle(0.4 pt); 
\filldraw (5.5,1) circle(0.4 pt); 
\filldraw (5.6,1) circle(0.4 pt); 

\draw [line width=2pt, line cap=rectengle, dash pattern=on 1pt off 2] (6,1) -- (7,1);
\path[edge] (6,1) -- (7,1);
\draw [line width=3pt, line cap=round, dash pattern=on 0pt off 1.3\pgflinewidth] (7,1) -- (8,1);
\path[edge] (7,1) -- (8,1);
\path[edge] (1,1) -- (1,0);
\path[edge] (2,1) -- (2,0);
\path[edge] (3,1) -- (3,0);
\path[edge] (4,1) -- (4,0);
\path[edge] (5,1) -- (5,0);
\path[edge] (7,1) -- (7,0);
\path[edge] (8,1) -- (8,0);
\path[edge] (6,1) -- (6,0);
\filldraw[black] (1,1) circle (1pt);
\filldraw[black] (3,1) circle (1pt);
\filldraw[black] (5,1) circle (1pt);
\filldraw[black] (6,1) circle (1pt);
\filldraw[black] (8,1) circle (1pt);
\filldraw[black] (1,0) circle (1pt);
\filldraw[black] (3,0) circle (1pt);
\filldraw[black] (5,0) circle (1pt);
\filldraw[black] (6,0) circle (1pt);
\filldraw[black] (8,0) circle (1pt);

\node (a1) at (1,1.2){$u_1^1$}; 
\node (a2) at (2,1.21) {$u_2^1$};  
\node (a3) at (3,1.21)  {$u_3^1$};  
\node (a4) at (4,1.21)  {$u_4^1$};  
\node (a5) at (5,1.21)  {$u_5^1$};  
\node (a6) at (6.1,1.21) {$u_{2s-1}^1$};  
\node (a7) at (7,1.21)  {$u_{2s}^1$};  
\node (a8) at (8,1.21)  {$u_1^1$};  

\node (a1) at (0.8,0.1){$u_1^2$}; 
\node (a2) at (1.8,0.21) {$u_2^2$};  
\node (a3) at (2.8,0.21)  {$u_3^2$};  
\node (a4) at (3.8,0.21)  {$u_4^2$};  
\node (a5) at (4.8,0.21)  {$u_5^2$};  
\node (a6) at (6.3,0.21) {$u_{2s-1}^2$};  
\node (a7) at (7.2,0.21)  {$u_{2s}^2$};  
\node (a8) at (8.2,0.1)  {$u_1^2$};  
   
\end{scope}

\begin{scope}[shift={(1,0)}]

\draw [line width=2pt, line cap=rectengle, dash pattern=on 1pt off 2] (1,1) -- (2,1);
\path[edge] (1,1) -- (2,1);

\draw [line width=3pt, line cap=round, dash pattern=on 0pt off 1.3\pgflinewidth] (2,1) -- (3,1);
\path[edge] (2,1) -- (3,1);

\draw [line width=2pt, line cap=rectengle, dash pattern=on 1pt off 2] (3,1) -- (4,1);
\path[edge] (3,1) -- (4,1);

\draw [line width=3pt, line cap=round, dash pattern=on 0pt off 1.3\pgflinewidth] (4,1) -- (5,1);
\path[edge] (4,1) -- (5,1);

\filldraw (5.4,1) circle(0.4 pt); 
\filldraw (5.5,1) circle(0.4 pt); 
\filldraw (5.6,1) circle(0.4 pt); 
\draw [line width=2pt, line cap=rectengle, dash pattern=on 1pt off 2] (6,1) -- (7,1);
\path[edge] (6,1) -- (7,1);
\draw [line width=3pt, line cap=round, dash pattern=on 0pt off 1.3\pgflinewidth] (7,1) -- (8,1);
\path[edge] (7,1) -- (8,1);
\path[edge,dashed] (1,1) -- (1,0);
\path[edge,dashed] (2,1) -- (2,0);
\path[edge,dashed] (3,1) -- (3,0);
\path[edge,dashed] (4,1) -- (4,0);
\path[edge,dashed] (5,1) -- (5,0);
\path[edge,dashed] (7,1) -- (7,0);
\path[edge,dashed] (8,1) -- (8,0);
\path[edge,dashed] (6,1) -- (6,0);

 \filldraw[black] (1,0) circle (1pt);
\filldraw[black] (3,0) circle (1pt);
\filldraw[black] (5,0) circle (1pt);
\filldraw[black] (6,0) circle (1pt);
\filldraw[black] (8,0) circle (1pt);
\node (a1) at (0.8,0.1){$u_1^3$}; 
\node (a2) at (1.8,0.21) {$u_2^3$};  
\node (a3) at (2.8,0.21)  {$u_3^3$};  
\node (a4) at (3.8,0.21)  {$u_4^3$};  
\node (a5) at (4.8,0.21)  {$u_5^3$};  
\node (a6) at (6.3,0.21) {$u_{2s-1}^3$};  
\node (a7) at (7.2,0.21)  {$u_{2s}^3$};  
\node (a8) at (8.2,0.1)  {$u_1^3$};

\end{scope}

\begin{scope}[shift={(1,-1)}]

\draw [line width=2pt, line cap=rectengle, dash pattern=on 1pt off 2] (1,1) -- (2,1);
\path[edge] (1,1) -- (2,1);

\draw [line width=3pt, line cap=round, dash pattern=on 0pt off 1.3\pgflinewidth] (2,1) -- (3,1);
\path[edge] (2,1) -- (3,1);

\draw [line width=2pt, line cap=rectengle, dash pattern=on 1pt off 2] (3,1) -- (4,1);
\path[edge] (3,1) -- (4,1);

\draw [line width=3pt, line cap=round, dash pattern=on 0pt off 1.3\pgflinewidth] (4,1) -- (5,1);
\path[edge] (4,1) -- (5,1);

\filldraw (5.4,1) circle(0.4 pt); 
\filldraw (5.5,1) circle(0.4 pt); 
\filldraw (5.6,1) circle(0.4 pt); 
\draw [line width=2pt, line cap=rectengle, dash pattern=on 1pt off 2] (6,1) -- (7,1);
\path[edge] (6,1) -- (7,1);
\draw [line width=3pt, line cap=round, dash pattern=on 0pt off 1.3\pgflinewidth] (7,1) -- (8,1);
\path[edge] (7,1) -- (8,1);
\filldraw (1,0.6) circle(0.4 pt); 
\filldraw (1,0.4) circle(0.4 pt); 
\filldraw (1,0.5) circle(0.4 pt);

\filldraw (2,0.6) circle(0.4 pt); 
\filldraw (2,0.4) circle(0.4 pt); 
\filldraw (2,0.5) circle(0.4 pt); 

\filldraw (3,0.6) circle(0.4 pt); 
\filldraw (3,0.4) circle(0.4 pt); 
\filldraw (3,0.5) circle(0.4 pt); 
 
\filldraw (4,0.6) circle(0.4 pt); 
\filldraw (4,0.4) circle(0.4 pt); 
\filldraw (4,0.5) circle(0.4 pt); 

\filldraw (5,0.6) circle(0.4 pt); 
\filldraw (5,0.4) circle(0.4 pt); 
\filldraw (5,0.5) circle(0.4 pt); 
 
\filldraw (7,0.6) circle(0.4 pt); 
\filldraw (7,0.4) circle(0.4 pt); 
\filldraw (7,0.5) circle(0.4 pt); 

\filldraw (8,0.6) circle(0.4 pt); 
\filldraw (8,0.4) circle(0.4 pt); 
\filldraw (8,0.5) circle(0.4 pt); 

\filldraw (6,0.6) circle(0.4 pt); 
\filldraw (6,0.4) circle(0.4 pt); 
\filldraw (6,0.5) circle(0.4 pt); 

\filldraw[black] (1,0) circle (1pt);
\filldraw[black] (3,0) circle (1pt);
\filldraw[black] (5,0) circle (1pt);
\filldraw[black] (6,0) circle (1pt);
\filldraw[black] (8,0) circle (1pt);

\node (a1) at (0.9,0.1){$u_1^{k-1}$}; 
\node (a2) at (2,0.2) {$u_2^{k-1}$};  
\node (a3) at (3,0.2)  {$u_3^{k-1}$};  
\node (a4) at (4,0.2)  {$u_4^{k-1}$};  
\node (a5) at (4.9,0.2)  {$u_5^{k-1}$};  
\node (a6) at (6.3,0.2) {$u_{2s-1}^{k-1}$};  
\node (a7) at (7.2,0.21)  {$u_{2s}^{k-1}$};  
\node (a8) at (8.2,0.18)  {$u_1^{k-1}$};  

\end{scope}

\begin{scope}[shift={(1,-2)}]

\draw [line width=2pt, line cap=rectengle, dash pattern=on 1pt off 2] (1,1) -- (2,1);
\path[edge] (1,1) -- (2,1);

\draw [line width=3pt, line cap=round, dash pattern=on 0pt off 1.3\pgflinewidth] (2,1) -- (3,1);
\path[edge] (2,1) -- (3,1);

\draw [line width=2pt, line cap=rectengle, dash pattern=on 1pt off 2] (3,1) -- (4,1);
\path[edge] (3,1) -- (4,1);

\draw [line width=3pt, line cap=round, dash pattern=on 0pt off 1.3\pgflinewidth] (4,1) -- (5,1);
\path[edge] (4,1) -- (5,1);

\filldraw (5.4,1) circle(0.4 pt); 
\filldraw (5.5,1) circle(0.4 pt); 
\filldraw (5.6,1) circle(0.4 pt); 
\draw [line width=2pt, line cap=rectengle, dash pattern=on 1pt off 2] (6,1) -- (7,1);
\path[edge] (6,1) -- (7,1);
\draw [line width=3pt, line cap=round, dash pattern=on 0pt off 1.3\pgflinewidth] (7,1) -- (8,1);
\path[edge] (7,1) -- (8,1);
\path[edge] (1,1) -- (1,0);
\path[edge] (2,1) -- (2,0);
\path[edge] (3,1) -- (3,0);
\path[edge] (4,1) -- (4,0);
\path[edge] (5,1) -- (5,0);
\path[edge] (7,1) -- (7,0);
\path[edge] (8,1) -- (8,0);
\path[edge] (6,1) -- (6,0);

 \filldraw[black] (1,0) circle (1pt);
\filldraw[black] (3,0) circle (1pt);
\filldraw[black] (5,0) circle (1pt);
\filldraw[black] (6,0) circle (1pt);
\filldraw[black] (8,0) circle (1pt);
\node (a1) at (0.8,0.1){$u_1^k$}; 
\node (a2) at (1.8,0.21) {$u_2^k$};  
\node (a3) at (2.8,0.21)  {$u_3^k$};  
\node (a4) at (3.8,0.21)  {$u_4^k$};  
\node (a5) at (4.8,0.21)  {$u_5^k$};  
\node (a6) at (6.3,0.21) {$u_{2s-1}^k$};  
\node (a7) at (7.2,0.21)  {$u_{2s}^k$};  
\node (a8) at (8.18,0.1)  {$u_1^k$};  

\end{scope}

\begin{scope}[shift={(1,-3)}]

\draw [line width=2pt, line cap=rectengle, dash pattern=on 1pt off 2] (1,1) -- (2,1);
\path[edge] (1,1) -- (2,1);

\draw [line width=3pt, line cap=round, dash pattern=on 0pt off 1.3\pgflinewidth] (2,1) -- (3,1);
\path[edge] (2,1) -- (3,1);

\draw [line width=2pt, line cap=rectengle, dash pattern=on 1pt off 2] (3,1) -- (4,1);
\path[edge] (3,1) -- (4,1);

\draw [line width=3pt, line cap=round, dash pattern=on 0pt off 1.3\pgflinewidth] (4,1) -- (5,1);
\path[edge] (4,1) -- (5,1);

\filldraw (5.4,1) circle(0.4 pt); 
\filldraw (5.5,1) circle(0.4 pt); 
\filldraw (5.6,1) circle(0.4 pt); 
\draw [line width=2pt, line cap=rectengle, dash pattern=on 1pt off 2] (6,1) -- (7,1);
\path[edge] (6,1) -- (7,1);
\draw [line width=3pt, line cap=round, dash pattern=on 0pt off 1.3\pgflinewidth] (7,1) -- (8,1);
\path[edge] (7,1) -- (8,1);
\path[edge,dashed] (1,1) -- (1,0);
\path[edge,dashed] (2,1) -- (2,0);
\path[edge,dashed] (3,1) -- (3,0);
\path[edge,dashed] (4,1) -- (4,0);
\path[edge,dashed] (5,1) -- (5,0);
\path[edge,dashed] (7,1) -- (7,0);
\path[edge,dashed] (8,1) -- (8,0);
\path[edge,dashed] (6,1) -- (6,0);

 \filldraw[black] (1,0) circle (1pt);
\filldraw[black] (3,0) circle (1pt);
\filldraw[black] (5,0) circle (1pt);
\filldraw[black] (6,0) circle (1pt);
\filldraw[black] (8,0) circle (1pt);
\node (a1) at (0.8,-0.2){$u_{1+2r}^1$}; 
\node (a2) at (1.9,-0.21) {$u_{2+2r}^1$};  
\node (a3) at (2.9,-0.21)  {$u_{3+2r}^1$};  
\node (a4) at (3.9,-0.21)  {$u_{4+2r}^1$};  
\node (a5) at (4.9,-0.21)  {$u_{5+2r}^1$};  
\node (a6) at (6.2,-0.21) {$u_{2r-1}^1$};  
\node (a7) at (7.2,-0.21)  {$u_{2r}^1$};  
\node (a8) at (8,-0.18)  {$u_{1+2r}^1$};  

\end{scope}

\begin{scope}[shift={(1,-4)}]

\draw [line width=2pt, line cap=rectengle, dash pattern=on 1pt off 2] (1,1) -- (2,1);
\path[edge] (1,1) -- (2,1);

\draw [line width=3pt, line cap=round, dash pattern=on 0pt off 1.3\pgflinewidth] (2,1) -- (3,1);
\path[edge] (2,1) -- (3,1);

\draw [line width=2pt, line cap=rectengle, dash pattern=on 1pt off 2] (3,1) -- (4,1);
\path[edge] (3,1) -- (4,1);

\draw [line width=3pt, line cap=round, dash pattern=on 0pt off 1.3\pgflinewidth] (4,1) -- (5,1);
\path[edge] (4,1) -- (5,1);

\filldraw (5.4,1) circle(0.4 pt); 
\filldraw (5.5,1) circle(0.4 pt); 
\filldraw (5.6,1) circle(0.4 pt);  
\draw [line width=2pt, line cap=rectengle, dash pattern=on 1pt off 2] (6,1) -- (7,1);
\path[edge] (6,1) -- (7,1);
\draw [line width=3pt, line cap=round, dash pattern=on 0pt off 1.3\pgflinewidth] (7,1) -- (8,1);
\path[edge] (7,1) -- (8,1);
\end{scope}

\begin{scope}[shift={(2,-4)}]
\draw [line width=2pt, line cap=rectengle, dash pattern=on 1pt off 2] (0,0) -- (1,0);
\path[edge] (0,0) -- (1,0);
\draw [line width=3pt, line cap=round, dash pattern=on 0pt off 1.3\pgflinewidth] (2,0) -- (3,0);
\path[edge] (2,0) -- (3,0);
\path[edge] (4,0) -- (5,0);
\path[edge, dashed] (6,0) -- (7,0);
\node () at (0.5,0.2){$0$}; 
\node () at (2.5,0.2){$2$}; 
\node () at (4.5,0.2){$1$}; 
\node () at (6.5,0.2){$3$}; 
\end{scope}

\end{tikzpicture}
\caption{A semi-equivelar gem of $(4^4)$-type  embedded regularly on $\mathbb{S}^1 \times \mathbb{S}^1$, where the addition in the subscript of $u^1_{i}$ is of modulo $2s$, $1 \leq i \leq 2s$ and $0\leq r \leq s-1$.}\label{fig:L(p,q)}
\end{figure}

\begin{corollary}\label{lemma:4^4}
There does not exist a $(4^4)$-type semi-equivelar gem representing $\mathbb{S}^2 \times \mathbb{S}^1$ embedded regularly on $\mathbb{S}^1 \times \mathbb{S}^1$.
\end{corollary}

\begin{corollary} \label{cor:}
There does not exist any $(4^4)$-type semi-equivelar gem  embedded regularly on $\mathbb{S}^1 \tilde{\times} \mathbb{S}^1$. 
\end{corollary}

\begin{proof}
To get $(4^4)$-type semi-equivelar non-bipartite gems $
\Gamma$ (up to isomorphism), the proof of Lemma \ref{newlemma:4^4} is modified as: $u_1^k$ is to be joined with $u_t^1$ by $3$-colored edge for some even $t$, $1 \leq t \leq 2s$. Then $u_2^k$ is joined to $u^1_{t-1}$ by $3$-colored edge, $u_3^k$ is joined to $u^1_{t-2}$ by $3$-colored edge, and so on. Thus, we get a $4$-regular colored graph $\Gamma$ embedded regularly on $\mathbb{S}^1 \tilde{\times} \mathbb{S}^1$. If $\frac{t}{2}$ is odd, then $\Gamma_{\{0,1,3\}}$ has a non-bipartite component generated by the vertices $u^l_{\frac{t}{2}}, u^l_{\frac{t}{2}+1}$, $1\leq l \leq k$. If $\frac{t}{2}$ is even, then $\Gamma_{\{1,2,3}\}$ has a non-bipartite component generated by the vertices $u^l_{\frac{t}{2}}, u^l_{\frac{t}{2}+1}$, $1\leq l \leq k$. In either case, we have a contradiction, as each component of $\Gamma_{\{0,1,3\}}$ and $\Gamma_{\{1,2,3\}}$ is bipartite since it represents the $2$-sphere. Therefore, the result follows.
\end{proof}
	
\begin{lemma}\label{lemma:L(p,q)}
For every pair of co-prime integers $(p,q)$, $p>q\geq 1$, there is a $(4^4)$-type semi-equivelar gem that represents the lens space $L(p,q)$ and is regularly embedded on $\mathbb{S}^1 \times \mathbb{S}^1$.
\end{lemma}

\begin{proof}
Consider a pair of co-prime integers $(p,q)$ with $p>q\geq 1$. By choosing $s=p$ and $r=q$ in the proof of Lemma \ref{newlemma:4^4}, we obtain a $(4^4)$-type semi-equivelar gem representing the lens space $L(p,q)$, regularly embedded on $\mathbb{S}^1 \times \mathbb{S}^1$ (refer to Figure \ref{fig:L(p,q)} for the case $(s,r)=(p, q)$). Additionally, for $k=2$, we obtain the standard $2p$-vertex gem that represents the lens space $L(p,q)$. This gem, well-documented in the literature, is commonly referred to as the standard crystallization of the lens space $L(p,q)$.
\end{proof}
	
\begin{lemma}\label{lemma:(6^2,4)}
There does not exist a $(6^2,4^1)$-type semi-equivelar gem embedded regularly on $\mathbb{RP}^2$.
\end{lemma}

\begin{proof}
If there exists a $(6^2,4^1)$-type semi-equivelar gem which embeds regularly on $\mathbb{RP}^2$ then its double cover must be a $(6^2,4^1)$-type semi-equivelar gem embedded regularly on $\mathbb{S}^2$. Now, Figure \ref{fig:no RP2} gives a $(6^2,4^1)$-type semi-equivelar gem $\Gamma$ embedded on $\mathbb{S}^2$. Moreover, $\Gamma$ is a unique $(6^2,4^1)$-type semi-equivelar gem, up to isomorphism, due to regular embedding of a $3$-colored gem. Now, antipodal identification of $\Gamma$, which is a truncated octahedron, does not preserve coloring adjacency. Thus, there does not exist any antipodal identification in $\Gamma$ to get a semi-equivelar embedding on $\mathbb{RP}^2$. Hence, the result follows.
\end{proof}

The Figures \ref{fig:1)}$(a)$, \ref{fig:1)}$(c)$, \ref{fig:1)}$(e)$ and \ref{fig:2}$(a)$ show the regular embeddings of semi-equivelar gems on the projective plane. Further, Figures \ref{fig:1)}$(b)$, \ref{fig:1)}$(d)$,  \ref{fig:1)}$(f)$, \ref{fig:2}$(b)$ and \ref{fig:no RP2} give the regular embeddings of semi-equivelar gems on $2$-sphere. Furthermore, Figures  \ref{fig:3)}$(a)$, \ref{fig:3)}$(b)$ and \ref{fig:3)}$(c)$  show the regular embeddings of semi-equivelar gems on the torus. It can be easily observed that changing the identifications on the boundary of $\mathbb{D}^2$ (rectangle) in Figures  \ref{fig:3)}$(a)$, \ref{fig:3)}$(b)$ and \ref{fig:3)}$(c)$  gives the semi-equivelar gems on Klein Bottle. Now, using Lemma \ref{lemma:(6^2,4)} and Corollary \ref{cor:}, we have the following theorem and thus answer the first problem.

\begin{theorem} 
Let $S$ be a surface with $\chi(S)\geq 0$. Then semi-equivelar gems embedded regularly on $S$ exist for every embedding type from Proposition $\ref{prop:dattasir}$ and Lemma $\ref{lemma:possibletypes}$, except for the $(6^2,4^1)$-type on real projective plane and the $(4^4)$-type on the Klein bottle.
\end{theorem}
	
\begin{remark}
{\rm
Let $M_1$ and $M_2$ be two PL-manifolds such that $\mathcal{G}(M_1) \neq \mathcal{G}(M_2)$. There may exist semi-equivelar gems representing $M_1$ and $M_2$ of the same type and same order embedded on the same surface. For example, $\mathbb{S}^3$ and $\mathbb{RP}^3$ have semi-equivelar gems of $(4^4)$-type and order $8$ embedded on the torus.}
\end{remark}
	
\begin{theorem}\label{seconda}
Let $S$ be a closed connected surface. Then $S$ admits a genus-minimal semi-equivelar gem. In particular, for $S=\#_n (\mathbb{S}^1 \times \mathbb{S}^1)$ (resp., $S=\#_n(\mathbb{RP}^2)$), the semi-equivelar gem of type $((4n+2)^3)$ (resp., $((2n+2)^3)$) exists.
\end{theorem}

\begin{proof}
In Figure \ref{fig:4}$(a)$, the semi-equivelar gem of type $((4n+2)^3)$ for the surface $\#_n (\mathbb{S}^1 \times \mathbb{S}^1)$ embedded on the surface $\#_n (\mathbb{S}^1 \times \mathbb{S}^1)$, has been constructed. The gem has $4n+2$ vertices (the vertices on the inside circle of the figure). Each pair of (dotted) $xy$ edges with a common mid-point is identified to get the surface.
		 
In Figure \ref{fig:4}$(b)$, the semi-equivelar gem of type $((2n+2)^3)$ for the surface $\#_n(\mathbb{RP}^2)$ embedded on the surface $\#_n(\mathbb{RP}^2)$,  has been constructed separately for even and odd $n$. The gem has $2n+2$ vertices (the vertices on the inside circles of the figure). Each pair of (dotted) $xy$ edges with a common mid-point is identified to get the surface.
		
The result follows with any of the semi-equivelar gems with zero genera representing $\mathbb{S}^2$ in Figures  \ref{fig:1)}$(b)$, \ref{fig:1)}$(d)$, \ref{fig:1)}$(f)$, \ref{fig:2}$(b)$ and \ref{fig:no RP2}. 
\end{proof}

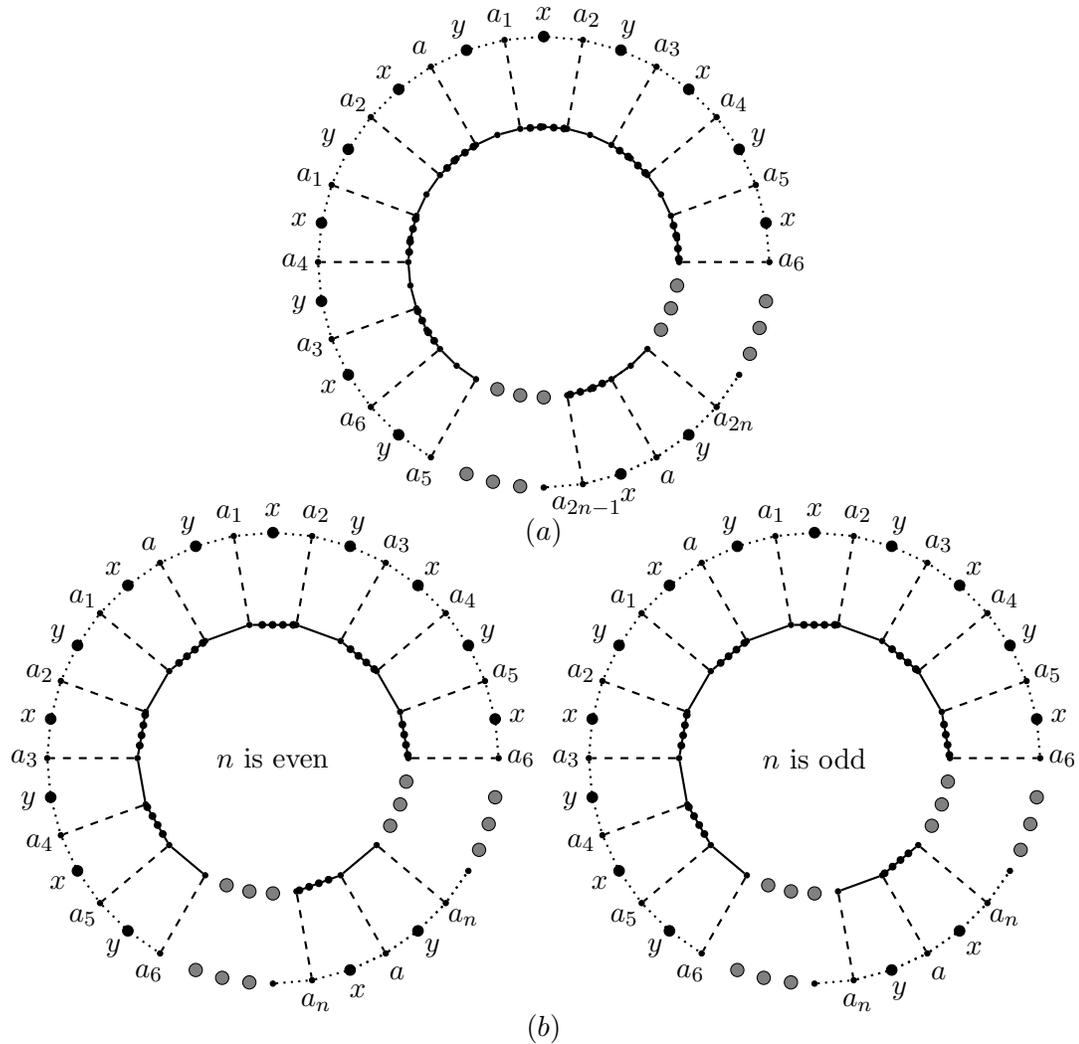
\begin{figure}[!h]
\tikzstyle{ver}=[]
\tikzstyle{vert}=[circle, draw, fill=black!100, inner sep=0pt, minimum width=2pt]
\tikzstyle{verti}=[circle, draw, fill=black!100, inner sep=0pt, minimum width=4pt]
\tikzstyle{vertii}=[circle, draw, fill=black!50, inner sep=0pt, minimum width=5pt]
\tikzstyle{edge} = [draw,thick,-]
\centering
					
\begin{tikzpicture}[scale=0.6]

\begin{scope}[shift={(8,0)}]
\foreach \x/\y in {270/x_7,250/y_7,260/13,350/x_9,340/17,330/y_9}{
\node[vertii] (\y) at (\x:3){};}
							
\foreach \x in {250,257,264,336,343,350}{
\node[vertii] (\x) at (\x:5){};}
							
\foreach \x/\y in {0/0,20/1,40/2,60/3,80/4,100/5,120/6,140/7,160/8,180/9,200/10,220/11,240/12,
280/14,300/15,320/16}{\node[vert] (\y) at (\x:3){};}
							
\foreach \x/\y in {0/0',20/1',40/2',60/3',80/4',100/5',120/6',140/7',160/8',180/9',200/10',220/11',240/12', 280/14',300/15',320/16',10/y_1',50/y_2',90/y_3',130/y_4',170/y_5',210/y_6',290/y_8', 30/x_1',70/x_2',110/x_3',150/x_4',190/x_5',230/x_6',310/x_8',270/x_7',330/y_9'}{\node[vert] (\y) at (\x:5){};}
							
\foreach \x/\y in {10/x,50/x,90/x,130/x,170/x,210/x}{
\node[ver] (\y) at (\x:5.5){$\y$};
\node[verti] (\y) at (\x:5){};
}
							
\foreach \x/\y in {30/y,70/y,110/y,150/y,190/y,230/y}{
\node[ver] (\y) at (\x:5.5){$\y$};
\node[verti] (\y) at (\x:5){};
}
							
\foreach \x/\y in {290/x,310/y}{\node[verti] (\y) at (\x:5){};}
							
\node[ver]  at (290:5.5){$y$};
\node[ver]  at (310:5.5){$x$};
							
							\foreach \x/\y in {0/a_{6},20/a_{5},40/a_{4},60/a_3,80/a_2,100/a_1,120/a,140/a_1,160/a_2,180/a_3,200/a_4,220/a_5, 240/a_6,
								280/a_{n},300/a,320/a_{n}}{
								\node[ver] (\y) at (\x:5.5){$\y$};
							
							}
							
					\foreach \x/\y in {0/1,1/2,2/3,3/4,4/5,5/6,6/7,7/8,8/9,9/10,10/11,11/12,14/15,15/16}{
								\path[edge] (\x) -- (\y);}
							\foreach \x/\y in {0/1,2/3,4/5,6/7,8/9,10/11,15/16}{
								\draw [line width=3pt, line cap=round, dash pattern=on 0pt off 1.3\pgflinewidth]  (\x) -- (\y);}
							
							\foreach \x/\y in {14'/y_8'}{
								\path[edge,dotted] (\x) -- (\y);}
							\foreach \x/\y in {0/0',1/1',2/2',3/3',4/4',5/5',6/6',7/7',8/8',9/9',10/10',11/11',12/12',14/14',15/15',16/16'}{
								\path[edge, dashed] (\x) -- (\y);}
							\foreach \x/\y in {0'/y_1',y_1'/1',1'/x_1',x_1'/2',2'/y_2',y_2'/3',3'/x_2',x_2'/4',4'/y_3',y_3'/5',5'/x_3',
								x_3'/6',6'/y_4',y_4'/7',7'/x_4',x_4'/8',8'/y_5',y_5'/9',9'/x_5',x_5'/10',10'/y_6',y_6'/11',11'/x_6',x_6'/12',14'/y_8',y_8'/15',15'/x_8',x_8'/16',16'/y_9',x_7'/14'}{
								\draw [edge,dotted] (\x) -- (\y);}

							\foreach \x/\y in {}{
								\path[edge] (\x) -- (\y);}
							
							\foreach \x/\y in {y_8'/15',15'/x_8',x_8'/16',16'/y_9'}{
								\path[thick,edge,dotted] (\x) -- (\y);}
							\node[ver] at (0,0){$n$ is odd};
							
						\end{scope}
						
						\begin{scope}[shift={(-4,0)}]
							\foreach \x/\y in {270/x_7,250/y_7,260/13,350/x_9,340/17,330/y_9}{
							
								\node[vertii] (\y) at (\x:3){};}
							
							\foreach \x in {250,257,264,336,343,350}{
							
								\node[vertii] (\x) at (\x:5){};}
							
							\foreach \x/\y in {0/0,20/1,40/2,60/3,80/4,100/5,120/6,140/7,160/8,180/9,200/10,220/11,240/12,
								280/14,300/15,320/16}{
							
								\node[vert] (\y) at (\x:3){};}
							
							\foreach \x/\y in {0/0',20/1',40/2',60/3',80/4',100/5',120/6',140/7',160/8',180/9',200/10',220/11',240/12',
								280/14',300/15',320/16',10/y_1',50/y_2',90/y_3',130/y_4',170/y_5',210/y_6',290/y_8', 30/x_1',70/x_2',110/x_3',150/x_4',190/x_5',230/x_6',310/x_8',270/x_7',330/y_9'}{
							
								\node[vert] (\y) at (\x:5){};}
							
							\foreach \x/\y in {10/x,50/x,90/x,130/x,170/x,210/x
							}{
								\node[ver] (\y) at (\x:5.5){$\y$};
								\node[verti] (\y) at (\x:5){};
							}
							
							\foreach \x/\y in {30/y,70/y,110/y,150/y,190/y,230/y}{
								\node[ver] (\y) at (\x:5.5){$\y$};
								\node[verti] (\y) at (\x:5){};
							}
							
							\foreach \x/\y in {290/x,310/y}{
								\node[verti] (\y) at (\x:5){};
							}
							
\node[ver]  at (290:5.5){$x$};
\node[ver]  at (310:5.5){$y$};
							
							\foreach \x/\y in {0/a_{6},20/a_{5},40/a_{4},60/a_3,80/a_2,100/a_1,120/a,140/a_1,160/a_2,180/a_3,200/a_4,220/a_5, 240/a_6,
								280/a_{n},300/a,320/a_{n}}{
								\node[ver] (\y) at (\x:5.5){$\y$};
								
							}

							\foreach \x/\y in {0/1,1/2,2/3,3/4,4/5,5/6,6/7,7/8,8/9,9/10,10/11,11/12,14/15,15/16}{
								\path[edge] (\x) -- (\y);}
							\foreach \x/\y in {0/1,2/3,4/5,6/7,8/9,10/11,14/15}{
								\draw [line width=3pt, line cap=round, dash pattern=on 0pt off 1.3\pgflinewidth]  (\x) -- (\y);}

							\foreach \x/\y in {0/0',1/1',2/2',3/3',4/4',5/5',6/6',7/7',8/8',9/9',10/10',11/11',12/12',14/14',15/15',16/16'}{
								\path[edge, dashed] (\x) -- (\y);}
							\foreach \x/\y in {0'/y_1',y_1'/1',1'/x_1',x_1'/2',2'/y_2',y_2'/3',3'/x_2',x_2'/4',4'/y_3',y_3'/5',5'/x_3',
								x_3'/6',6'/y_4',y_4'/7',7'/x_4',x_4'/8',8'/y_5',y_5'/9',9'/x_5',x_5'/10',10'/y_6',y_6'/11',11'/x_6',x_6'/12',14'/y_8',y_8'/15',15'/x_8',x_8'/16',16'/y_9',x_7'/14'}{
								\draw [edge,dotted] (\x) -- (\y);}

							\foreach \x/\y in {y_8'/15',15'/x_8',x_8'/16',16'/y_9'}{
								\path[thick,edge,dotted] (\x) -- (\y);}
							\node[ver] at (0,0){$n$ is even};
		\node[ver] at (6,-6){$(b)$};					
						\end{scope}
						
						\begin{scope}[shift={(2,11)}]
							\foreach \x/\y in {270/x_7,250/y_7,260/13,350/x_9,340/17,330/y_9}{
								
								\node[vertii] (\y) at (\x:3){};}
							
							\foreach \x in {250,257,264,336,343,350}{
								
								\node[vertii] (\x) at (\x:5){};}
							
							\foreach \x/\y in {0/0,20/1,40/2,60/3,80/4,100/5,120/6,140/7,160/8,180/9,200/10,220/11,240/12,
								280/14,300/15,320/16,10/y_1,50/y_2,90/y_3,130/y_4,170/y_5,210/y_6,290/y_8, 30/x_1,70/x_2,110/x_3,150/x_4,190/x_5,230/x_6,310/x_8}{
							
								\node[vert] (\y) at (\x:3){};}
							
							\foreach \x/\y in {0/0',20/1',40/2',60/3',80/4',100/5',120/6',140/7',160/8',180/9',200/10',220/11',240/12',
								280/14',300/15',320/16',10/y_1',50/y_2',90/y_3',130/y_4',170/y_5',210/y_6',290/y_8', 30/x_1',70/x_2',110/x_3',150/x_4',190/x_5',230/x_6',310/x_8',270/x_7',330/y_9'}{
								
								\node[vert] (\y) at (\x:5){};}
							
							\foreach \x/\y in {10/x,50/x,90/x,130/x,170/x,210/x
							}{
								\node[ver] (\y) at (\x:5.5){$\y$};
								\node[verti] (\y) at (\x:5){};
							}
							
							\foreach \x/\y in {30/y,70/y,110/y,150/y,190/y,230/y}{
								\node[ver] (\y) at (\x:5.5){$\y$};
								\node[verti] (\y) at (\x:5){};
							}
							
							\foreach \x/\y in {290/x,310/y}{
								\node[verti] (\y) at (\x:5){};
							}
							
\node[ver]  at (290:5.5){$x$};
\node[ver]  at (310:5.5){$y$};
							
							\foreach \x/\y in {0/a_{6},20/a_{5},40/a_{4},60/a_3,80/a_2,100/a_1,120/a,140/a_2,160/a_1,180/a_4,200/a_3,220/a_6, 240/a_5,
								320/a_{2n},300/a}{
								\node[ver] (\y) at (\x:5.5){$\y$};
								
							}
							\node[ver] () at (280:5.5){$a_{2n-1}$};

							\foreach \x/\y in {0/y_1,y_1/1,1/x_1,x_1/2,2/y_2,y_2/3,3/x_2,x_2/4,4/y_3,y_3/5,5/x_3,
								x_3/6,6/y_4,y_4/7,7/x_4,x_4/8,8/y_5,y_5/9,9/x_5,x_5/10,10/y_6,y_6/11,11/x_6,x_6/12,14/y_8,y_8/15,15/x_8,x_8/16}{
								\path[edge] (\x) -- (\y);}
							\foreach \x/\y in {0/y_1,y_1/1,2/y_2,y_2/3,4/y_3,y_3/5,6/y_4,y_4/7,8/y_5,y_5/9,
								10/y_6,y_6/11,14/y_8,y_8/15}{
								\draw [line width=3pt, line cap=round, dash pattern=on 0pt off 1.3\pgflinewidth]  (\x) -- (\y);}
							
							\foreach \x/\y in {14/y_8,y_8/15,14'/y_8'}{
								\path[edge,dotted] (\x) -- (\y);}
							\foreach \x/\y in {0/0',1/1',2/2',3/3',4/4',5/5',6/6',7/7',8/8',9/9',10/10',11/11',12/12',14/14',15/15',16/16'}{
								\path[edge, dashed] (\x) -- (\y);}
							\foreach \x/\y in {0'/y_1',y_1'/1',1'/x_1',x_1'/2',2'/y_2',y_2'/3',3'/x_2',x_2'/4',4'/y_3',y_3'/5',5'/x_3',
								x_3'/6',6'/y_4',y_4'/7',7'/x_4',x_4'/8',8'/y_5',y_5'/9',9'/x_5',x_5'/10',10'/y_6',y_6'/11',11'/x_6',x_6'/12',14'/y_8',y_8'/15',15'/x_8',x_8'/16',16'/y_9',x_7'/14'}{
								\draw [edge,dotted] (\x) -- (\y);}

							\foreach \x/\y in {}{
								\path[edge] (\x) -- (\y);}
							
							\foreach \x/\y in {y_8'/15',15'/x_8',x_8'/16',16'/y_9',15/x_8,x_8/16,15/x_8}{
								\path[thick,edge,dotted] (\x) -- (\y);}
							
		\node[ver] at (0,-6){$(a)$};					
						\end{scope}

					\end{tikzpicture}
\caption{$(a)$ Embedding on $\#_n (\mathbb{S}^1\times \mathbb{S}^1)$ of gem representing $\#_n (\mathbb{S}^1\times \mathbb{S}^1)$ of type $((4n+2)^3)$,
$(b)$ Embedding on $\#_n \mathbb{RP}^2$ of gem representing $\#_{n} \mathbb{RP}^2$ of type $((2n+2)^3)$.}\label{fig:4}
\end{figure}

\begin{theorem}\label{theorem:secondbiff}
Let $M$ be a closed connected orientable $d$-manifold (where $d \geq 3$) such that $\mathcal{G}(M) \leq 1$. Then, $M$ admits a genus-minimal semi-equivelar gem if and only if $M$ is a lens space.
\end{theorem} 

\begin{proof}
It has been established that the only PL $d$–manifold (for $d \geq 2$) with regular genus zero is $\mathbb{S}^d$, as documented in \cite{fg82}. By examining the possible embedding types on the $2$-sphere, as presented in Lemma \ref{lemma:possibletypes}, it is evident that a semi-equivelar gem representing $\mathbb{S}^d$ cannot be genus-minimal for $d\geq 3$. Likewise, an inspection of the possible embedding types on the torus, as outlined in Proposition \ref{prop:dattasir}, reveals that a semi-equivelar gem representing a closed orientable PL $d$-manifold $M$ with $\mathcal{G}(M) = 1$ cannot be genus-minimal for $d\geq 4$. Therefore, if a closed orientable PL $d$–manifold $M$ (for $d \geq 3$) with $\mathcal{G}(M) \leq 1$ possesses a genus-minimal semi-equivelar gem, then $\mathcal{G}(M) = 1$ and $d=3$. In \cite{h76}, it is demonstrated that every closed orientable $3$-manifold with $\mathcal{G}(M)=1$ is homeomorphic to either $\mathbb{S}^2 \times \mathbb{S}^1$ or a lens space $L(p,q)$. Further, Corollary \ref{lemma:4^4} asserts that $\mathbb{S}^2 \times \mathbb{S}^1$ does not admit a genus-minimal semi-equivelar gem. Since Lemma \ref{lemma:L(p,q)} confirms that every lens space $L(p,q)$ (for co-prime integers $p>q\geq 1$) admits a genus-minimal semi-equivelar gem, this concludes the proof.
\end{proof}

The term `semi-equivelar gem' is inspired by semi-equivelar maps, which do not have concept of bigons. Getting motivated by the previous work on maps, we have chosen to exclude bigons in our approach. However, it is worth noting that 2-length cycles in a graph are of great significance in crystallization theory. This consideration has led us to introduce the term `semi-equivelar gems (with bigons)'. Furthermore, the absence of genus-minimal semi-equivelar gems for certain spaces has motivated our focus on semi-equivelar gems (with bigons). We are considering future research related to semi-equivelar gems (with bigons).

\begin{remark}
{\rm
If we extend our consideration to include $2$-gons, which are polygons of length $2$, during the embedding of a gem on a surface, then a gem meeting the conditions specified in Definition \ref{def:semiequivelar} is denoted as a semi-equivelar gem (with bigons). Additionally, a PL $d$-manifold $M$ represented by a semi-equivelar gem (with bigons) is said to admit a semi-equivelar gem (with bigons). It is important to note that a semi-equivelar gem is inherently a semi-equivelar gem (with bigons).}
\end{remark}
  		
\begin{remark}\label{remark:secondb}
{\rm It is known that the only PL $d$–manifold (for $d \geq 2$) with regular genus zero is $\mathbb{S}^d$, as established in \cite{fg82}. The standard crystallization of $\mathbb{S}^d$ has exactly two vertices, where each $ij$-colored cycle is a $2$-gon for $i, j \in \Delta_d$. This configuration serves as a genus-minimal semi-equivelar gem (with bigons) representing $\mathbb{S}^d$. In \cite{h76}, it has been established that every closed orientable $3$-manifold with $\mathcal{G}(M)=1$ is homeomorphic to either $\mathbb{S}^2 \times \mathbb{S}^1$ or a lens space $L(p,q)$. In \cite{c89, ch93, cs93}, it is established that for every closed connected orientable PL $d$-manifold $M$ with $d \geq 4$ and $\mathcal{G}(M)=1$, it holds that $M$ is PL-homeomorphic to $\mathbb{S}^{d-1} \times \mathbb{S}^1$. Figure \ref{fig:3)}$(d)$ gives a genus-minimal semi-equivelar gem (with bigons) representing $\mathbb{S}^{d-1} \times \mathbb{S}^1$, $d\geq 3$. It follows from Lemma \ref{lemma:L(p,q)}, every lens space admits a genus-minimal semi-equivelar gem.  Therefore, for a closed connected orientable $d$-manifold $M$  (where $d \geq 3$) with $\mathcal{G}(M)\leq 1$, $M$ admits a genus-minimal semi-equivelar gem (with bigons). Furthermore, each surface $S$ admits a genus-minimal semi-equivelar gem (with bigons) because Theorem \ref{seconda} implies that $S$ admits a genus-minimal semi-equivelar gem. In Figure \ref{fig:3)}$(d)$, if we reverse the orientation of any two opposite edges of the rectangle, then we obtain a genus-minimal semi-equivelar gem (with bigons) representing $\mathbb{S}^{d-1} \tilde{\times} \mathbb{S}^1$.}
\end{remark}

Theorems \ref{seconda}, \ref{theorem:secondbiff} and Remark \ref{remark:secondb} provide answers to the second problem for a subcollection of closed PL $d$-manifolds in terms of the existence of semi-equivelar gems without and with bigons respectively. The next natural question to ask is:

\smallskip

\noindent \textbf{Question 1.} Does every closed PL $d$-manifold admit a semi-equivelar gem (with bigons)? 

\smallskip
		
\noindent If the answer is affirmative then we can define a new term {\it Semi-equivelar} genus as follows.

\begin{definition}
Let $M$ be a closed PL $d$-manifold. The {\it Semi-equivelar genus} (with bigons) $\mathcal G_{seq}(M)$ of $M$ is defined as
$$\mathcal G_{seq}(M) = \min \{\rho(\Gamma) \ | \  (\Gamma,\gamma)\in \mathbb{G}_d \mbox{ is a semi-equivelar gem (with bigons) representing } M\}.$$
\end{definition}
It is clear from the definition of the regular genus $\mathcal G(M)$ and the semi-equivelar genus (with bigons) $\mathcal G_{seq}(M)$ of a closed PL $d$-manifold $M$ that $\mathcal G_{seq}(M) \geq \mathcal G(M)$. This article demonstrates that given a surface or a closed orientable $d$-manifold (where $d \geq 3$) $M$ with $\mathcal G(M) \leq 1$, it follows that $\mathcal G_{seq}(M) = \mathcal G(M)$.
Furthermore, we pose the following question:

\smallskip
		
\noindent \textbf{Question 2.} Let $M$ be a closed PL $d$-manifold. Does $\mathcal G_{seq}(M) =\mathcal G(M)$ hold?

			\begin{figure}[!h]
				\tikzstyle{ver}=[]
				\tikzstyle{vert}=[circle, draw, fill=black!100, inner sep=0pt, minimum width=2pt]
				\tikzstyle{vertii}=[circle, draw, fill=black!50, inner sep=0pt, minimum width=4pt]
				\tikzstyle{edge} = [draw,thick,-]
				\centering
				
				\begin{tikzpicture}[scale=0.35]
					
					\begin{scope}[shift={(12,0)}, scale=0.6]
						\foreach \x/\y/\z in {-5/0/0,-3/2/1,-1/4/2,1/4/3,3/2/4,5/0/5,3/-2/6,1/-4/7,-1/-4/8,-3/-2/9}
						{
							\node[vert] (\z) at (\x,\y){};
						} 
						\foreach \x/\y in {0/1,1/2,2/3,3/4,4/5,5/6,6/7,7/8,8/9,9/0}{\path[edge] (\x) -- (\y);}
						\foreach \x/\y in {0/1,2/3,4/5,6/7,8/9}{\draw [line width=2.5pt, line cap=round, dash pattern=on 0pt off 2\pgflinewidth]  (\x) -- (\y);}
						
						\foreach \x/\y/\z in {-5/-4/10,-3/-6/11,-1/-8/12,1/-8/13,3/-6/14,5/-4/15,7/-4/16,9/-2/17,7/0/18,5/2/19,
							5/4/20,3/6/21,1/6/22,-1/6/23,-3/6/24,-5/4/25,-5/2/26,-7/0/27,-9/-2/28,-7/-4/29,
							-9/-6/30,-11/-8/31,-9/-10/32,-7/-12/33,-5/-12/34,-3/-10/35,3/-10/36,5/-12/37,7/-12/38,
							9/-10/39,11/-8/40,9/-6/41,11/0/42,13/2/43,13/4/44,11/6/45,9/8/46,7/6/47,5/8/48,
							3/10/49,1/10/50,-1/10/51,-3/10/52,-5/8/53,-7/6/54,-9/8/55,-11/6/56,-13/4/57,-13/2/58,
							-11/0/59}
						{
							\node[vert] (\z) at (\x,\y){};
						} 
						
						\foreach \x/\y in {10/29,29/30,30/31,31/32,32/33,33/34,34/35,35/12,12/11,11/10,
							13/14,14/15,15/16,16/41,41/40,40/39,39/38,38/37,37/36,36/13,
							17/18,18/19,19/20,20/47,47/46,46/45,45/44,44/43,43/42,42/17,
							21/22,22/23,23/24,24/53,53/52,52/51,51/50,50/49,49/48,48/21,
							25/26,26/27,27/28,28/59,59/58,58/57,57/56,56/55,55/54,54/25}{\path[edge] (\x) -- (\y);}
						
						\foreach \x/\y in {11/10,35/12,33/34,31/32,29/30,
							14/15,16/41,40/39,38/37,36/13, 18/19,20/47,46/45,44/43,42/17,
							22/23,24/53,52/51,50/49,48/21, 26/27,28/59,58/57,56/55,54/25}{\draw [line width=2.5pt, line cap=round, dash pattern=on 0pt off 2\pgflinewidth]  (\x) -- (\y);}
						
						\foreach \x/\y in {0/27,1/26,2/23,3/22,4/19,5/18,6/15,7/14,8/11,9/10,12/13,16/17,20/21,
							24/25,28/29,30/59,35/36,41/42,47/48,53/54}{\path[edge,dashed] (\x) -- (\y);}
						\foreach \x/\y/\z in {-5/-4/10,-3/-6/11,-1/-8/12,1/-8/13,3/-6/14,5/-4/15,7/-4/16,9/-2/17,7/0/18,5/2/19,
							5/4/20,3/6/21,1/6/22,-1/6/23,-3/6/24,-5/4/25,-5/2/26,-7/0/27,-9/-2/28,-7/-4/29,
							-13/-6/61,-11/-8/31,-11/-10/62,-13/-12/63,-5/-12/34,-3/-14/66,3/-10/36,5/-12/37,7/-12/38,
							9/-10/39,11/-8/40,9/-6/41,11/0/42,13/2/43,13/4/44,15/6/74,9/12/78,7/6/47,5/8/48,
							5/12/79,1/14/81,-1/10/51,-3/10/52,-5/8/53,-7/6/54,-9/8/55,-11/6/56,-11/8/86,-13/10/87,-13/2/58,
							-15/0/60,-11/-14/64,-9/-14/65,
							15/0/73,13/-6/72,11/-10/71,13/-12/70,11/-14/69,9/-14/68,3/-14/67,15/8/75,-15/6/89,
							-15/8/88,13/10/76,11/8/77,-9/12/85,-5/12/84,-3/12/83,3/12/80,-1/14/82,
							-17/10/90,17/10/99,-15/12/91,15/12/98,-3/16/94,3/16/95,-17/14/92,17/14/97,-15/16/93,
							15/16/96,15/-7/102,17/-1/101,17/6/100,-15/-7/107,-17/-1/108,-17/6/109,-8/-16/106,
							8/-16/103,-3/-16/105,3/-16/104, -9/18/110, 9/18/119, -19/10/111,19/10/118,-19/4/112,
							19/4/117,-19/-3/113,19/-3/116,-5/-18/114,5/-18/115}{
							
							\node[vert] (\z) at (\x,\y){};
						} 
						
						\foreach \x/\y in {60/61,61/62,72/73,71/72,68/69,70/71,62/63,64/65,66/67,65/66,67/68,
							60/89,73/74,74/75,76/77,86/87,88/89,77/78,85/86,84/85,78/79,79/80,83/84,80/81,82/83,
							88/90,75/99,87/91,76/98,82/94,81/95,91/92,92/93,93/94,95/96,96/97,97/98,99/100,100/101,
							101/102,102/70,90/109,109/108,108/107,107/63,106/64,69/103,106/105,105/104,104/103,
							110/119,111/112,118/117,112/113,117/116,114/115}{\path[edge] (\x) -- (\y);}
						
						\foreach \x/\y in {61/62,71/72,65/66,67/68,60/89,73/74,77/78,85/86,79/80,83/84,88/90,87/91,75/99,76/98,
							82/94,81/95,92/93,96/97,102/70,100/101,109/108,107/63,106/64,69/103,105/104,
							112/113,117/116,114/115}{\draw [line width=2.5pt, line cap=round, dash pattern=on 0pt off 2\pgflinewidth]  (\x) -- (\y);}

						\foreach \x/\y in {58/60,61/31,62/32,43/73,40/72,39/71,63/64,33/65,38/68,69/70,34/66,37/67,57/89,44/74,
							87/88,56/86,45/77,75/76,55/85,46/78,52/84,49/79,51/83,50/80,81/82,98/99,90/91,94/95,
							93/110,96/119,92/111,97/118,100/117,109/112,108/113,101/116,105/114,104/115}{\path[edge,dashed] (\x) -- (\y);}
						
						\draw[edge,dashed] plot [smooth,tension=1] coordinates{(107) (-13,-14) (106)};
						\draw[edge,dashed] plot [smooth,tension=1] coordinates{(102) (13,-14) (103)};
						
						\draw[edge] plot [smooth,tension=1] coordinates{(110) (-16.5,16.5) (111)};
						\draw[line width=3pt, line cap=round, dash pattern=on 0pt off 2\pgflinewidth] plot [smooth,tension=1] coordinates{(110) (-16.5,16.5) (111)};
						
						\draw[edge] plot [smooth,tension=1] coordinates{(118) (16.5,16.5) (119)};
						\draw[line width=3pt, line cap=round, dash pattern=on 0pt off 2\pgflinewidth] plot [smooth,tension=1] coordinates{(118) (16.5,16.5) (119)};

						\draw[edge] plot [smooth,tension=1] coordinates{(113) (-15,-15) (114)};\draw[edge] plot [smooth,tension=1] coordinates{(115) (15,-15) (116)};
						\node[ver] at (0,-20){$(b)$};
						
					\end{scope}
					\begin{scope}[shift={(-11,0)}, scale=0.6]
						\foreach \x/\y/\z in {-5/0/0,-3/2/1,-1/4/2,1/4/3,3/2/4,5/0/5,3/-2/6,1/-4/7,-1/-4/8,-3/-2/9}
						{
							\node[vert] (\z) at (\x,\y){};
						} 
						
						\foreach \x/\y in {0/1,1/2,2/3,3/4,4/5,5/6,6/7,7/8,8/9,9/0}{\path[edge] (\x) -- (\y);}
						\foreach \x/\y in {0/1,2/3,4/5,6/7,8/9}{\draw [line width=2.5pt, line cap=round, dash pattern=on 0pt off 2\pgflinewidth]  (\x) -- (\y);}
						
						\foreach \x/\y/\z in {-5/-4/10,-3/-6/11,-1/-8/12,1/-8/13,3/-6/14,5/-4/15,7/-4/16,9/-2/17,7/0/18,5/2/19,
							5/4/20,3/6/21,1/6/22,-1/6/23,-3/6/24,-5/4/25,-5/2/26,-7/0/27,-9/-2/28,-7/-4/29,
							-9/-6/30,-11/-8/31,-9/-10/32,-7/-12/33,-5/-12/34,-3/-10/35,3/-10/36,5/-12/37,7/-12/38,
							9/-10/39,11/-8/40,9/-6/41,13/-2/42,13/2/43,13/4/44,13/8/45,9/8/46,7/6/47,5/8/48,
							3/10/49,1/10/50,-1/10/51,-3/10/52,-5/8/53,-7/6/54,-9/8/55,-13/8/56,-13/4/57,-13/2/58,
							-13/-2/59}
						{
							\node[vert] (\z) at (\x,\y){};
						} 
						
						\foreach \x/\y in {10/29,29/30,30/31,31/32,32/33,33/34,34/35,35/12,12/11,11/10,
							13/14,14/15,15/16,16/41,41/40,40/39,39/38,38/37,37/36,36/13,
							17/18,18/19,19/20,20/47,47/46,46/45,45/44,44/43,43/42,42/17,
							21/22,22/23,23/24,24/53,53/52,52/51,51/50,50/49,49/48,48/21,
							25/26,26/27,27/28,28/59,59/58,58/57,57/56,56/55,55/54,54/25}{\path[edge] (\x) -- (\y);}
						
						\foreach \x/\y in {11/10,35/12,33/34,31/32,29/30,
							14/15,16/41,40/39,38/37,36/13, 18/19,20/47,46/45,44/43,42/17,
							22/23,24/53,52/51,50/49,48/21, 26/27,28/59,58/57,56/55,54/25}{\draw [line width=2.5pt, line cap=round, dash pattern=on 0pt off 2\pgflinewidth]  (\x) -- (\y);}
						
						\foreach \x/\y in {0/27,1/26,2/23,3/22,4/19,5/18,6/15,7/14,8/11,9/10,12/13,16/17,20/21,
							24/25,28/29,30/59,35/36,41/42,47/48,53/54}{\path[edge,dashed] (\x) -- (\y);}

						\foreach \x/\y/\z in {-15/12/c_1,15/12/c_2,15/-14/c_3,-15/-14/c_4}
						{
							\node[vert] (\z) at (\x,\y){};
						} 
						\foreach \x/\y in {c_1/c_2,c_2/c_3,c_3/c_4,c_4/c_1}{\draw[edge,dotted] (\x) -- (\y);}
						
						\foreach \x/\y/\z in {-5/12/b_52, -3/12/b_51, 3/12/b_50, 5/12/b_49, 13/12/b_46, 
							15/10/b_45, 15/2/b_44, 15/0/b_43, 15/-4/b_40,3/-14/b_37, 9/-14/b_38, 13/-14/b_39,
							-3/-14/b_34,-9/-14/b_33, -13/-14/b_32, -15/2/b_57, -15/0/b_58, -15/-4/b_31, -13/12/b_55, -15/10/b_56}
						{
							\node[vert] (\z) at (\x,\y){};
						} 
						\foreach \x/\y in {52/b_52,51/b_51,50/b_50,49/b_49,46/b_46,45/b_45,44/b_44,43/b_43,
							40/b_40,37/b_37,38/b_38,39/b_39,34/b_34,33/b_33,32/b_32, 31/b_31,57/b_57,58/b_58, 55/b_55,56/b_56}{\draw[edge,dashed] (\x) -- (\y);}

						\draw [-{Stealth[scale=2]}] (0,12) -- (1,12); 
						\draw [-{Stealth[scale=2]}] (1,-14)--(0,-14); 
						\draw [-{Stealth[scale=2]}] (-15,-0.5)--(-15,0); 
						\draw [-{Stealth[scale=2]}] (-15,-1.5)--(-15,-1); 
						\draw [-{Stealth[scale=2]}] (15,-0.4)--(15,-0.9); 
						\draw [-{Stealth[scale=2]}] (15,-1.4)--(15,-1.9); 
						\node[ver] at (0,-16){$(a)$};
						
					\end{scope}	
				\end{tikzpicture}
				
\smallskip
				
\caption{$(a)$ Embedding on $\mathbb{RP}^2$ of gem representing $\mathbb{RP}^2$  of type $(4,6,10)$ and $(b)$ Embedding on $\mathbb{S}^2$ of gem representing $\mathbb{S}^2$  of type $(4,6,10)$.}\label{fig:2}
\end{figure}
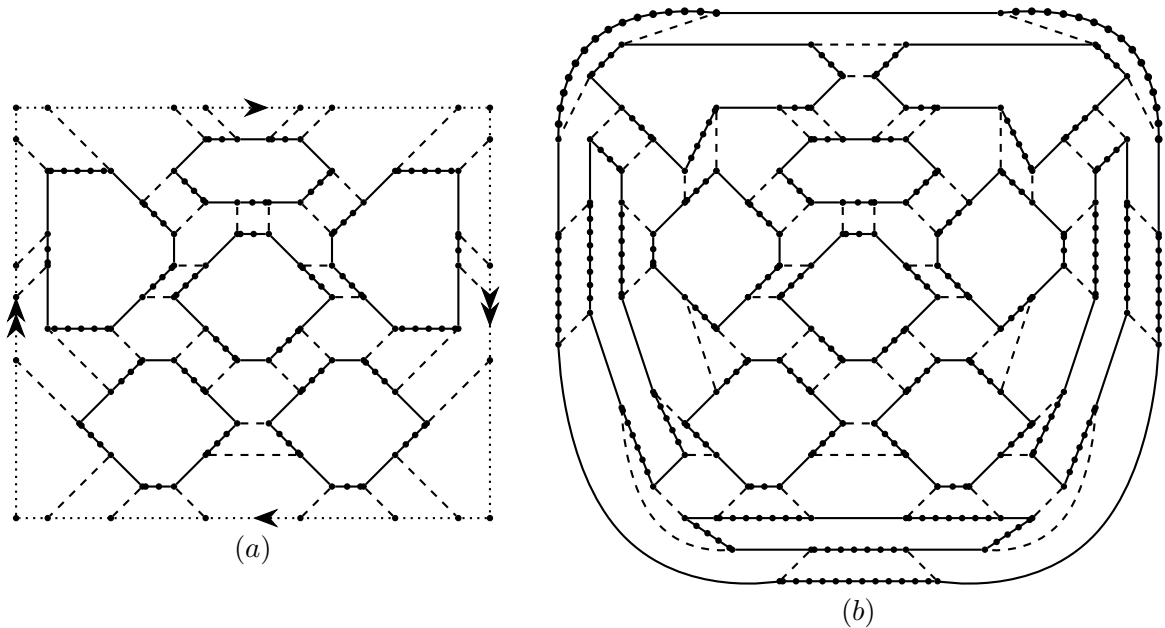

\bigskip
			
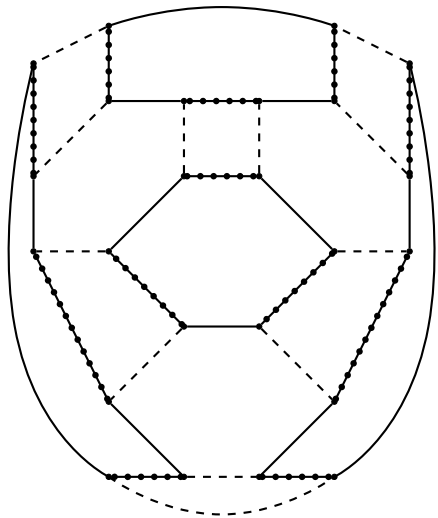
\begin{figure}[!h]
				\tikzstyle{ver}=[]
				\tikzstyle{vert}=[circle, draw, fill=black!100, inner sep=0pt, minimum width=2pt]
				\tikzstyle{vertii}=[circle, draw, fill=black!50, inner sep=0pt, minimum width=4pt]
				\tikzstyle{edge} = [draw,thick,-]
				\centering
				
				\begin{tikzpicture}[scale=0.5]

					\begin{scope}[shift={(0,0)}]
						\foreach \x/\y/\z in {-3/0/0,-1/2/1,1/2/2,3/0/3,1/-2/4,-1/-2/5,-3/-4/6,-1/-6/7,1/-6/8,3/-4/9,5/0/10,
							5/2/11,3/4/12,1/4/13,-1/4/14,-3/4/15,-5/2/16,-5/0/17,-3/-6/18,3/-6/19,5/5/20,3/6/21,
							-3/6/22,-5/5/23}
						{
							\node[vert] (\z) at (\x,\y){};
						} 
						
						\foreach \x/\y in {0/1,1/2,2/3,3/4,4/5,5/0,6/7,7/18,23/16,16/17,17/6,12/13,13/14,14/15,15/22,21/12,8/9,9/10,10/11,11/20,19/8}{\path[edge] (\x) -- (\y);}
						\foreach \x/\y in {1/2,3/4,5/0,16/23,6/17,7/18,8/19,9/10,11/20,13/14,15/22,12/21}{\draw [line width=2.5pt, line cap=round, dash pattern=on 0pt off 2\pgflinewidth]  (\x) -- (\y);}
						\foreach \x/\y in {0/17,1/14,2/13,3/10,4/9,5/6,7/8,11/12,15/16,22/23,20/21}{\path[edge,dashed] (\x) -- (\y);}
						
						\draw[edge,dashed] plot [smooth,tension=1] coordinates{(18) (0,-7) (19)};
						\draw[edge] plot [smooth,tension=1] coordinates{(19) (5.5,-2) (20)};
						\draw[edge] plot [smooth,tension=1] coordinates{(18) (-5.5,-2) (23)};
						\draw[edge] plot [smooth,tension=1] coordinates{(21) (0,6.5) (22)};
						\end{scope}
					\end{tikzpicture}
					
\bigskip
\caption{ Embedding on $\mathbb{S}^2$ of gem representing $\mathbb{S}^2$  of type $(6^2,4^1)$}\label{fig:no RP2}
\end{figure}
			
\newpage

\begin{figure}[!h]
			\tikzstyle{ver}=[]
			\tikzstyle{vert}=[circle, draw, fill=black!100, inner sep=0pt, minimum width=1pt]
			\tikzstyle{vertii}=[circle, draw, fill=black!50, inner sep=0pt, minimum width=4pt]
			\tikzstyle{edge} = [draw,thick,-]
			\centering
			
			\begin{tikzpicture}[scale=0.3]
				\begin{scope}[shift={(-14,0)}]
					\foreach \x/\y/\z in {-3/0/0,0/3/1,3/0/2,0/-3/3}
					{
						\node[vert] (\z) at (\x,\y){};
					} 
					
					\foreach \x/\y in {0/1,1/2,2/3,3/0}{\path[edge] (\x) -- (\y);}
					\foreach \x/\y in {0/1,2/3}{\draw [line width=2.5pt, line cap=round, dash pattern=on 0pt off 2\pgflinewidth]  (\x) -- (\y);}

					\foreach \x/\y/\z in {-6/6/c_1,6/6/c_2,6/-6/c_3,-6/-6/c_4}
					{
						\node[vert] (\z) at (\x,\y){};
					} 
					\foreach \x/\y in {c_1/c_2,c_2/c_3,c_3/c_4,c_4/c_1}{\draw[edge,dotted] (\x) -- (\y);}
					
					\foreach \x/\y/\z in {0/6/b_1,6/0/b_2,0/-6/b_3,-6/0/b_4}
					{
						\node[vert] (\z) at (\x,\y){};
					} 
					\foreach \x/\y in {0/b_4,1/b_1,2/b_2,3/b_3}{\path[edge,dashed] (\x) -- (\y);}
					\draw [-{Stealth[scale=2]}] (0,6) -- (0.5,6); 
					\draw [-{Stealth[scale=2]}] (0.2,-6)--(-0.2,-6); 
					\draw [-{Stealth[scale=2]}] (-6,0)--(-6,0.2); 
					\draw [-{Stealth[scale=2]}] (-6,0.6)--(-6,0.8); 
					\draw [-{Stealth[scale=2]}] (6,0.5)--(6,0); 
					\draw [-{Stealth[scale=2]}] (6,-0.3)--(6,-0.9); 
					\node[ver] at (0,-8){$(a)$};
				\end{scope}

				\begin{scope}[scale=1.5,shift={(9,0)}]
					\foreach \x/\y/\z in {-2/0/0,0/2/1,2/0/2,0/-2/3,0/-4/4,4/0/5,0/4/6,-4/0/7}
					{
						\node[vert] (\z) at (\x,\y){};
					} 
					
					\foreach \x/\y in {0/1,1/2,2/3,3/0,4/5,5/6,6/7,7/4}{\path[edge] (\x) -- (\y);}
					\foreach \x/\y in {0/1,2/3,4/5,6/7}{\draw [line width=2.5pt, line cap=round, dash pattern=on 0pt off 2\pgflinewidth]  (\x) -- (\y);}
					\foreach \x/\y in {0/7,1/6,2/5,3/4}{\path[edge,dashed] (\x) -- (\y);}

						\node[ver] at (0,-6){$(b)$};
						
					\end{scope}
					
					\begin{scope}[shift={(-12,-22)}]
						
						\foreach \x/\y/\z in {-3/1/0,-1/3/1,1/3/2,3/1/3,3/-1/4,1/-3/5,-1/-3/6,-3/-1/7}
						{
							\node[vert] (\z) at (\x,\y){};
						} 
						
						\foreach \x/\y in {0/1,1/2,2/3,3/4,4/5,5/6,6/7,7/0}{\path[edge] (\x) -- (\y);}
						\foreach \x/\y in {1/2,3/4,5/6,7/0}{\draw [line width=2.5pt, line cap=round, dash pattern=on 0pt off 2\pgflinewidth]  (\x) -- (\y);}

						\foreach \x/\y/\z in {-5/-1/8,-7/-3/9,-5/-5/10,-3/-5/11,3/-5/12,5/-5/13,7/-3/14,5/-1/15,5/1/16,7/3/17,
							5/5/18,3/5/19,-3/5/20,-5/5/21,-7/3/22,-5/1/23}
						{
							\node[vert] (\z) at (\x,\y){};} 
						\foreach \x/\y in {8/9,10/11,11/12,12/13,14/15,15/16,16/17,18/19,19/20,20/21,22/23,23/8}{\path[edge] (\x) -- (\y);}
						\foreach \x/\y in {11/12,15/16,19/20,23/8}{\draw [line width=2.5pt, line cap=round, dash pattern=on 0pt off 2\pgflinewidth]  (\x) -- (\y);}
						
						\foreach \x/\y/\z in {-9/3/b_1, -9/-3/b_2,-5/-7/b_3,5/-7/b_4,9/-3/b_5,9/3/b_6,5/7/b_7,-5/7/b_8,-9/7/c_1,-9/-7/c_2,9/-7/c_3,9/7/c_4}
						{
							\node[vert] (\z) at (\x,\y){};} 
						
						\foreach \x/\y in {b_1/22,b_2/9,b_3/10,b_4/13,b_5/14,b_6/17,b_7/18,b_8/21}{\path[edge] (\x) -- (\y);}
						\foreach \x/\y in {b_1/22,b_2/9,b_3/10,b_4/13,b_5/14,b_6/17,b_7/18,b_8/21}{\draw [line width=2.5pt, line cap=round, dash pattern=on 0pt off 2\pgflinewidth]  (\x) -- (\y);}
						\foreach \x/\y in {c_1/c_2,c_2/c_3,c_3/c_4,c_4/c_1}{\draw[edge,dotted] (\x) -- (\y);}
						\foreach \x/\y in {0/23,1/20,2/19,3/16,4/15,5/12,6/11,7/8,9/10,13/14,17/18,21/22}{\path[edge,dashed] (\x) -- (\y);}

						\draw [-{Stealth[scale=2]}] (0,7) -- (0.5,7); 
						\draw [-{Stealth[scale=2]}] (0.5,-7)--(0,-7); 
						\draw [-{Stealth[scale=2]}] (-9,0.1)--(-9,0.3); 
						\draw [-{Stealth[scale=2]}] (-9,0.7)--(-9,0.9); 
						\draw [-{Stealth[scale=2]}] (9,0.9)--(9,0.7); 
						\draw [-{Stealth[scale=2]}] (9,0.3)--(9,0.1); 
						\node[ver] at (0,-9){$(c)$};
					\end{scope}

					\begin{scope}[shift={(14,-22)}]
						\foreach \x/\y/\z in {-3/1/0,-1/3/1,1/3/2,3/1/3,3/-1/4,1/-3/5,-1/-3/6,-3/-1/7}
						{
							\node[vert] (\z) at (\x,\y){};
						} 
						
						\foreach \x/\y in {0/1,1/2,2/3,3/4,4/5,5/6,6/7,7/0}{\path[edge] (\x) -- (\y);}
						\foreach \x/\y in {1/2,3/4,5/6,7/0}{\draw [line width=2.5pt, line cap=round, dash pattern=on 0pt off 2\pgflinewidth]  (\x) -- (\y);}

						\foreach \x/\y/\z in {-5/-1/8,-7/-3/9,-5/-5/10,-3/-5/11,3/-5/12,5/-5/13,7/-3/14,5/-1/15,5/1/16,7/3/17,
							5/5/18,3/5/19,-3/5/20,-5/5/21,-7/3/22,-5/1/23}
						{
							\node[vert] (\z) at (\x,\y){};} 
						\foreach \x/\y in {8/9,10/11,11/12,12/13,14/15,15/16,16/17,18/19,19/20,20/21,22/23,23/8}{\path[edge] (\x) -- (\y);}
						\foreach \x/\y in {11/12,15/16,19/20,23/8}{\draw [line width=2.5pt, line cap=round, dash pattern=on 0pt off 2\pgflinewidth]  (\x) -- (\y);}
						
						\foreach \x/\y/\z in {-9/3/24,-11/1/25,-11/-1/26, -9/-3/27,-5/-7/28, -3/-7/29,3/-7/30,  5/-7/31,9/-3/32,11/-1/33,11/1/34,9/3/35,5/7/36,3/7/37,-3/7/38,-5/7/39}
						{
							\node[vert] (\z) at (\x,\y){};} 
						
						\foreach \x/\y in {10/28,28/29,29/30,30/31,31/13,14/32,32/33,33/34,34/35,35/17,
							18/36,36/37,37/38,38/39,39/21,22/24,24/25,25/26,26/27,27/9}{\path[edge] (\x) -- (\y);}
						\foreach \x/\y in {10/28,29/30,31/13,14/32,33/34,35/17, 18/36,37/38,39/21,22/24,25/26,27/9}{\draw [line width=2.5pt, line cap=round, dash pattern=on 0pt off 2\pgflinewidth]  (\x) -- (\y);}

						\foreach \x/\y/\z in {-13/1/40,-13/-1/41,  -3/-9/42,3/-9/43,  13/-1/44,13/1/45,3/9/46,-3/9/47}
						{
							\node[vert] (\z) at (\x,\y){};}
						
						\foreach \x/\y in {0/23,1/20,2/19,3/16,4/15,5/12,6/11,7/8,9/10,13/14,17/18,21/22,24/39,35/36,27/28,31/32,
							25/40,26/41,29/42,30/43,33/44,34/45,37/46,38/47}{\path[edge,dashed] (\x) -- (\y);}
						
						\foreach \x/\y in {40/41,42/43,44/45,46/47}{\path[edge] (\x) -- (\y);}
						\foreach \x/\y in {40/41,42/43,44/45,46/47}{\draw [line width=2.5pt, line cap=round, dash pattern=on 0pt off 2\pgflinewidth]  (\x) -- (\y);}
						\draw[edge] plot [smooth,tension=1] coordinates{(41) (-8,-6.7) (42)};
						\draw[edge] plot [smooth,tension=1] coordinates{(43) (8,-6.7) (44)};
						\draw[edge] plot [smooth,tension=1] coordinates{(45) (8,6.7) (46)};
						\draw[edge] plot [smooth,tension=1] coordinates{(47) (-8,6.7) (40)};

						\node[ver] at (0,-11){$(d)$};
					\end{scope}
					
					\begin{scope}[shift={(-12,-42)}]
						\foreach \x/\y/\z in {-5/1/0,-5/3/1,-3/5/2,-1/5/3,1/5/4,3/5/5,5/3/6,5/1/7,5/-1/8,5/-3/9,3/-5/10,1/-5/11,
							-1/-5/12,-3/-5/13,-5/-3/14,-5/-1/15}
						{
							\node[vert] (\z) at (\x,\y){};
						} 
						
						\foreach \x/\y/\z in {-0.6/5/u_1,0/5/u_2,0.6/5/u_3,
							-0.6/-5/u_1,0/-5/u_2,0.6/-5/u_3,
							5/-0.6/a_1,5/0/a_2,5/0.6/a_2,
							-5/-0.6/a_1,-5/0/a_2,-5/0.6/a_2}
						{
							\node[vertii] (\z) at (\x,\y){};
						} 
						
						\foreach \x/\y in {0/1,1/2,2/3,4/5,5/6,6/7,8/9,9/10,10/11,12/13,13/14,14/15}{\path[edge] (\x) -- (\y);}
						\foreach \x/\y in {1/2,5/6,9/10,13/14}{\draw [line width=2.5pt, line cap=round, dash pattern=on 0pt off 2\pgflinewidth]  (\x) -- (\y);}

						\foreach \x/\y/\z in {-7/7/c_1,7/7/c_2,7/-7/c_3,-7/-7/c_4}
						{
							\node[vert] (\z) at (\x,\y){};
						} 
						\foreach \x/\y in {c_1/c_2,c_2/c_3,c_3/c_4,c_4/c_1}{\draw[edge,dotted] (\x) -- (\y);}

						\foreach \x/\y/\z in {-7/1/b_0,-7/3/b_1,-3/7/b_2,-1/7/b_3,1/7/b_4,3/7/b_5,7/3/b_6,7/1/b_7,7/-1/b_8,7/-3/b_9,3/-7/b_10,1/-7/b_11,-1/-7/b_12,-3/-7/b_13,-7/-3/b_14,-7/-1/b_15}
						{
							\node[vert] (\z) at (\x,\y){};} 
						\foreach \x/\y in {0/b_0,1/b_1,2/b_2,3/b_3,4/b_4,5/b_5,6/b_6,7/b_7,8/b_8,9/b_9,10/b_10,11/b_11,12/b_12,13/b_13,14/b_14,15/b_15}{\path[edge,dashed] (\x) -- (\y);}
						
						\draw [-{Stealth[scale=2]}] (-0.2,7) -- (0.7,7); 
						\draw [-{Stealth[scale=2]}] (0.2,-7)--(-0.4,-7); 
						\draw [-{Stealth[scale=2]}] (-7,0)--(-7,0.2); 
						\draw [-{Stealth[scale=2]}] (-7,0.6)--(-7,0.8); 
						\draw [-{Stealth[scale=2]}] (7,0.5)--(7,0.3); 
						\draw [-{Stealth[scale=2]}] (7,-0.1)--(7,-0.3); 
						\node[ver] at (0,-9){$(e)$};
					\end{scope}
					
					\begin{scope}[shift={(14,-42)}]
						\foreach \x/\y/\z in {-5/1/0,-5/3/1,-3/5/2,-1/5/3,1/5/4,3/5/5,5/3/6,5/1/7,5/-1/8,5/-3/9,3/-5/10,1/-5/11,
							-1/-5/12,-3/-5/13,-5/-3/14,-5/-1/15}
						{
							\node[vert] (\z) at (\x,\y){};
						} 
						
						\foreach \x/\y/\z in {-0.6/5/u_1,0/5/u_2,0.6/5/u_3,-0.6/-5/u_1,0/-5/u_2,0.6/-5/u_3,
							5/-0.6/a_1,5/0/a_2,5/0.6/a_2,
							-5/-0.6/a_1,-5/0/a_2,-5/0.6/a_2}
						{
							\node[vertii] (\z) at (\x,\y){};
						}
						
						\foreach \x/\y/\z in {-0.6/7/u_1,0/7/u_2,0.6/7/u_3,-0.6/-7/u_1,0/-7/u_2,0.6/-7/u_3,
							7/-0.6/a_1,7/0/a_2,7/0.6/a_2,
							-7/-0.6/a_1,-7/0/a_2,-7/0.6/a_2}
						{
							\node[vertii] (\z) at (\x,\y){};
						}

						\foreach \x/\y in {0/1,2/3,4/5,6/7,8/9,10/11,12/13,14/15}{\path[edge] (\x) -- (\y);}

						\foreach \x/\y/\z in {-7/1/b_0,-7/3/b_1,-3/7/b_2,-1/7/b_3,1/7/b_4,3/7/b_5,7/3/b_6,7/1/b_7,7/-1/b_8,7/-3/b_9,3/-7/b_10,1/-7/b_11,-1/-7/b_12,-3/-7/b_13,-7/-3/b_14,-7/-1/b_15}
						{
							\node[vert] (\z) at (\x,\y){};} 
						\foreach \x/\y in {0/b_0,1/b_1,2/b_2,3/b_3,4/b_4,5/b_5,6/b_6,7/b_7,8/b_8,9/b_9,10/b_10,11/b_11,12/b_12,13/b_13,14/b_14,15/b_15}{\path[edge,dashed] (\x) -- (\y);}
						
						\draw[edge] plot [smooth,tension=1] coordinates{(1) (-4.3,4.3) (2)};
						\draw[line width=3pt, line cap=round, dash pattern=on 0pt off 2\pgflinewidth] plot [smooth,tension=1] coordinates{(1) (-4.3,4.3) (2)};
						
						\draw[edge] plot [smooth,tension=1] coordinates{(5) (4.3,4.3) (6)};
						\draw[line width=3pt, line cap=round, dash pattern=on 0pt off 2\pgflinewidth] plot [smooth,tension=1] coordinates{(5) (4.3,4.3) (6)};
						
						\draw[edge] plot [smooth,tension=1] coordinates{(9) (4.3,-4.3) (10)};
						\draw[line width=3pt, line cap=round, dash pattern=on 0pt off 2\pgflinewidth] plot [smooth,tension=1] coordinates{(9) (4.3,-4.3) (10)};
						
						\draw[edge] plot [smooth,tension=1] coordinates{(13) (-4.3,-4.3) (14)};
						\draw[line width=3pt, line cap=round, dash pattern=on 0pt off 2\pgflinewidth] plot [smooth,tension=1] coordinates{(13) (-4.3,-4.3) (14)};
						
						\draw[edge] plot [smooth,tension=1] coordinates{(b_1) (-5.5,5.5) (b_2)};
						\draw[line width=3pt, line cap=round, dash pattern=on 0pt off 2\pgflinewidth] plot [smooth,tension=1] coordinates{(b_1) (-5.5,5.5) (b_2)};
						
						\draw[edge] plot [smooth,tension=1] coordinates{(b_5) (5.5,5.5) (b_6)};
						\draw[line width=3pt, line cap=round, dash pattern=on 0pt off 2\pgflinewidth] plot [smooth,tension=1] coordinates{(b_5) (5.5,5.5) (b_6)};
						
						\draw[edge] plot [smooth,tension=1] coordinates{(b_9) (5.5,-5.5) (b_10)};
						\draw[line width=3pt, line cap=round, dash pattern=on 0pt off 2\pgflinewidth] plot [smooth,tension=1] coordinates{(b_9) (5.5,-5.5) (b_10)};
						
						\draw[edge] plot [smooth,tension=1] coordinates{(b_13) (-5.5,-5.5) (b_14)};
						\draw[line width=3pt, line cap=round, dash pattern=on 0pt off 2\pgflinewidth] plot [smooth,tension=1] coordinates{(b_13) (-5.5,-5.5) (b_14)};
						
						\foreach \x/\y in {b_2/b_3,b_4/b_5,b_6/b_7,b_8/b_9,b_10/b_11,b_12/b_13,b_14/b_15,b_0/b_1}{\path[edge] (\x) -- (\y);}
						
						\node[ver] at (0,-9){$(f)$};
					\end{scope}

				\end{tikzpicture}
\caption{$(a)$  Embedding on $\mathbb{RP}^2$ of gem representing $\mathbb{RP}^2$  of type $(4^3)$, $(b)$ Embedding on $\mathbb{S}^2$ of gem representing $\mathbb{S}^2$  of type $(4^3)$, $(c)$ Embedding on $\mathbb{RP}^2$ of gem representing $\mathbb{RP}^2$  of type  $(4,6,8)$, $(d)$  Embedding on $\mathbb{S}^2$ of gem representing $\mathbb{S}^2$  of type $(4,6,8)$, $(e)$ Embedding on $\mathbb{RP}^2$ of gem representing $\mathbb{RP}^2$  of type $(4^2,2p)$ and $(f)$ Embedding on $\mathbb{S}^2$ of gem representing $\mathbb{S}^2$  of type $(4^2,2p)$.}\label{fig:1)}
\end{figure}
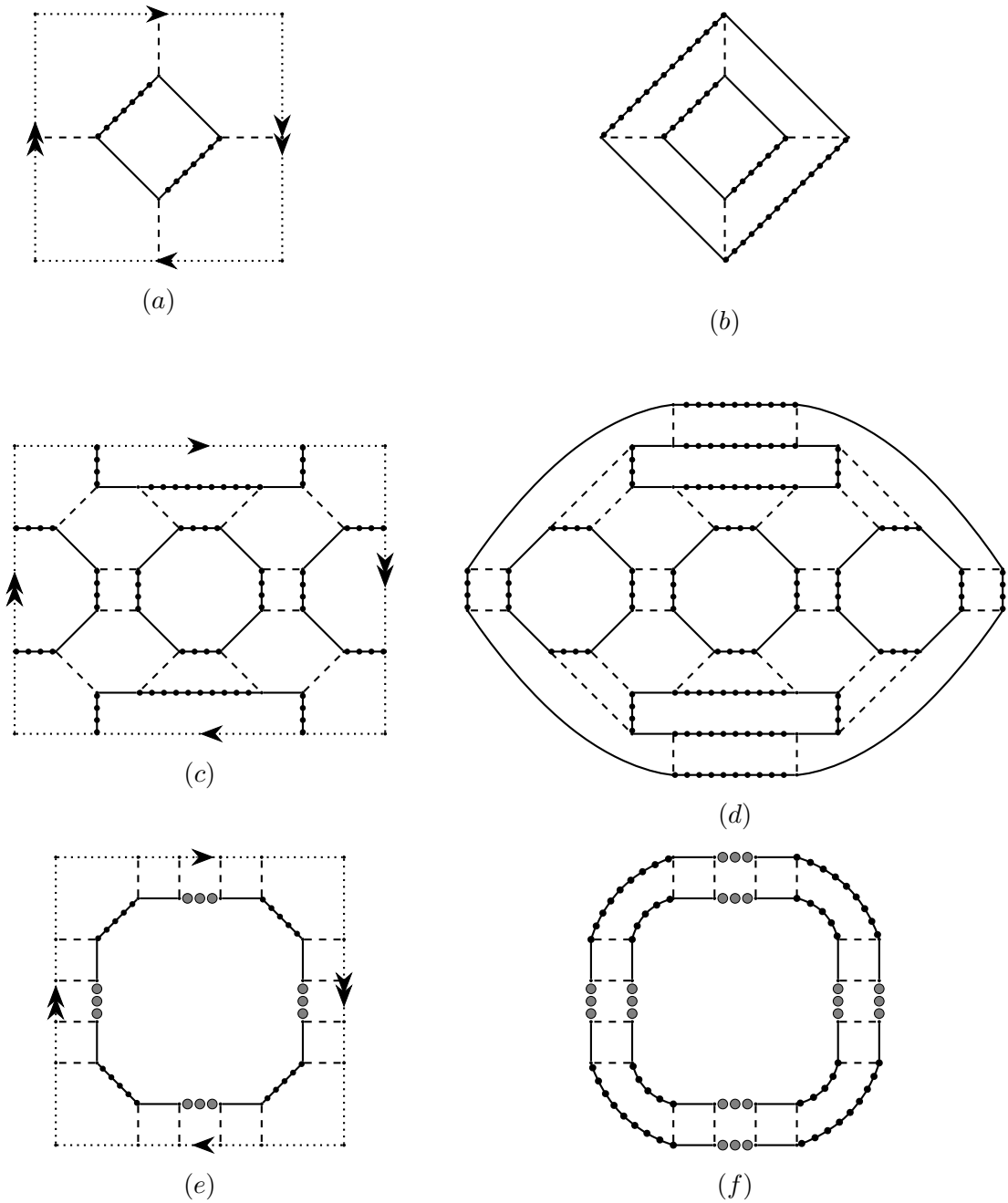

\newpage
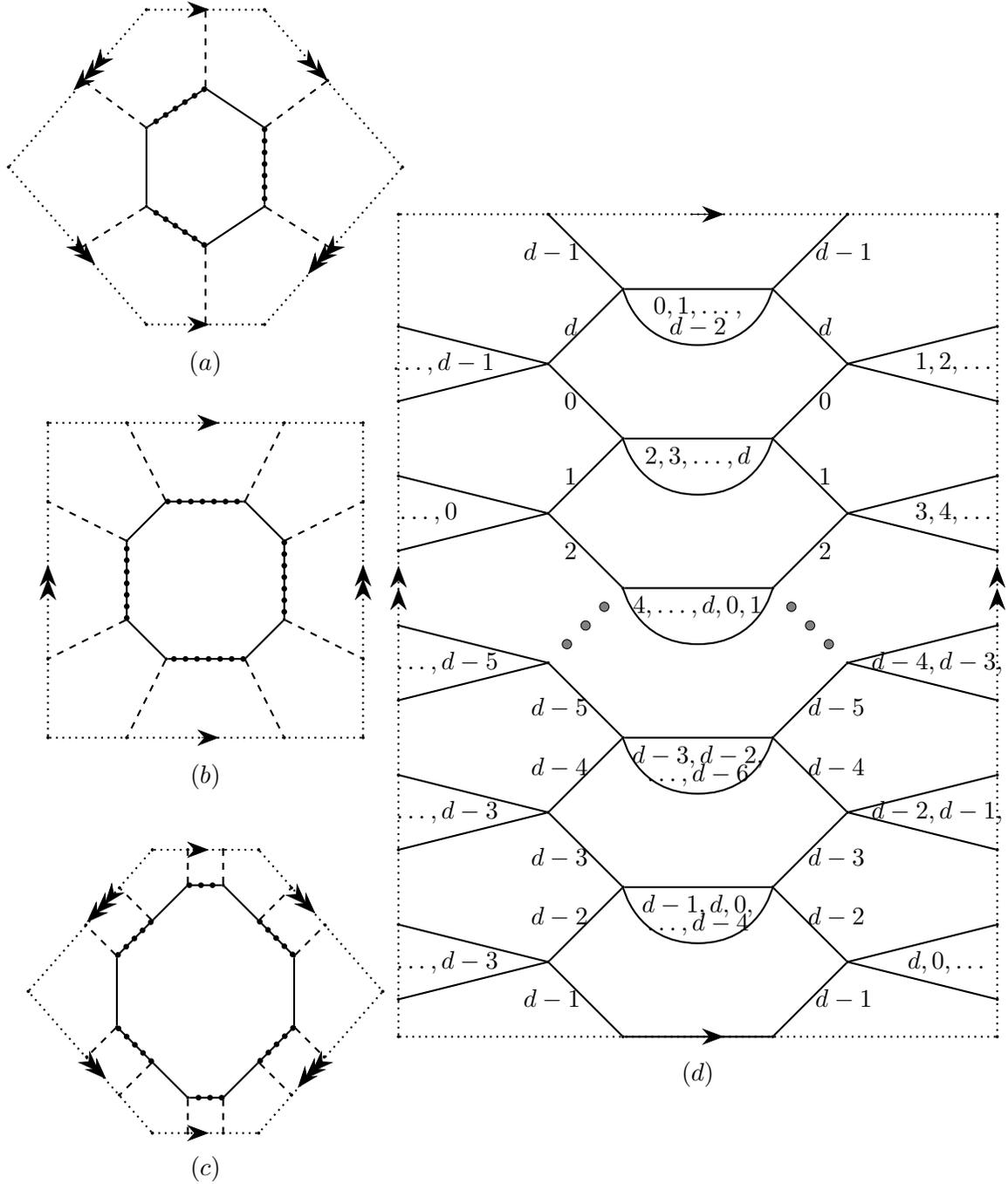
\begin{figure}[!h]
					\tikzstyle{ver}=[]
					\tikzstyle{vert}=[circle, draw, fill=black!100, inner sep=0pt, minimum width=1pt]
					\tikzstyle{vertii}=[circle, draw, fill=black!50, inner sep=0pt, minimum width=4pt]
					\tikzstyle{edge} = [draw,thick,-]
					\centering
					
					\begin{tikzpicture}[scale=0.3]

						\begin{scope}[shift={(-14,0)}]
							\foreach \x/\y/\z in {0/3/0,3/1/1,3/-3/2,0/-5/3,-3/-3/4,-3/1/5,-3/7/c_0,3/7/c_1,-3/-9/c_4,3/-9/c_3,10/-1/c_2,-10/-1/c_5,0/7/b_0,0/-9/b_3,6.4/-5.15/b_2,-6.4/-5.15/b_4,6.2/3.4/b_1,-6.2/3.4/b_5}
							{
								\node[vert] (\z) at (\x,\y){};
							} 
							
							\foreach \x/\y in {0/1,1/2,2/3,3/4,4/5,5/0}{\path[edge] (\x) -- (\y);}
							\foreach \x/\y in {1/2,3/4,0/5}{\draw [line width=2.5pt, line cap=round, dash pattern=on 0pt off 2\pgflinewidth]  (\x) -- (\y);}
							\foreach \x/\y in {0/b_0,3/b_3,2/b_2,4/b_4,5/b_5,1/b_1}{\path[edge,dashed] (\x) -- (\y);}
							\foreach \x/\y in {c_0/c_1,c_1/c_2,c_2/c_3,c_3/c_4,c_4/c_5,c_5/c_0}{\path[edge,dotted] (\x) -- (\y);}
							\draw [-{Stealth[scale=2]}] (-0.2,7) -- (0.2,7); 
							\draw [-{Stealth[scale=2]}] (-0.2,-9)--(0.2,-9); 
							\draw [-{Stealth[scale=2]}] (5.5,4.1)--(5.7,3.9); 
							\draw [-{Stealth[scale=2]}] (6,3.6)--(6.2,3.4);

							\draw [-{Stealth[scale=2]}] (-6.4,-5.15)--(-6,-5.55); 
							\draw [-{Stealth[scale=2]}] (-6,-5.55)--(-5.6,-5.95); 
							
							\draw [-{Stealth[scale=2]}] (-5.5,4.15)--(-5.7,3.9); 
							\draw [-{Stealth[scale=2]}] (-6,3.6)--(-6.2,3.4); 
							\draw [-{Stealth[scale=2]}] (-6.5,3.1)--(-6.7,2.9); 
							
							\draw [-{Stealth[scale=2]}] (6.4,-5.15)--(6,-5.55); 
							\draw [-{Stealth[scale=2]}] (6,-5.55)--(5.6,-5.95); 
							\draw [-{Stealth[scale=2]}] (5.6,-5.95)--(5.2,-6.45); 
							\node[ver] at (0,-11){$(a)$};
						\end{scope}
						\begin{scope}[shift={(-14,-22)}]
							\foreach \x/\y/\z in {-2/4/0,2/4/1,4/2/2,4/-2/3,2/-4/4,-2/-4/5,-4/-2/6,-4/2/7}
							{
								\node[vert] (\z) at (\x,\y){};
							}

							\foreach \x/\y/\z in {-8/8/c_1,8/8/c_2, 8/-8/c_3,-8/-8/c_4}
							{
								\node[vert] (\z) at (\x,\y){};
							} 
							\foreach \x/\y in {c_1/c_2,c_2/c_3,c_3/c_4,c_4/c_1}{\draw[edge,dotted] (\x) -- (\y);}
							
							\foreach \x/\y/\z in {-4/8/b_0,4/8/b_1,8/4/b_2,8/-4/b_3,-4/-8/b_5,4/-8/b_4,-8/-4/b_6,-8/4/b_7}
							{
								\node[vert] (\z) at (\x,\y){};
							} 
							
							\foreach \x/\y in {0/1,1/2,2/3,3/4,4/5,5/6,6/7,7/0}{\path[edge] (\x) -- (\y);}
							\foreach \x/\y in {0/1,2/3,4/5,6/7}{\draw [line width=2.5pt, line cap=round, dash pattern=on 0pt off 2\pgflinewidth]  (\x) -- (\y);}
							\foreach \x/\y in {0/b_0,1/b_1,2/b_2,3/b_3,4/b_4,5/b_5,6/b_6,7/b_7}{\path[edge,dashed] (\x) -- (\y);}

							\draw [-{Stealth[scale=2]}] (-0.2,8) -- (0.7,8); 
							
							\draw [-{Stealth[scale=2]}] (-0.2,-8) -- (0.7,-8); 
							
							\draw [-{Stealth[scale=2]}] (-8,-0.2) -- (-8,0.1); 
							\draw [-{Stealth[scale=2]}] (-8,0.5) -- (-8,0.8); 
							
							\draw [-{Stealth[scale=2]}] (8,-0.2) -- (8,0.1); 
							\draw [-{Stealth[scale=2]}] (8,0.5) -- (8,0.8); 
							\node[ver] at (0,-10){$(b)$};
						\end{scope}
						
						\begin{scope}[shift={(-14,-42)}, scale=0.9]
							\foreach \x/\y/\z in {-1/5/0,1/5/1,3/3/2,5/1/3,5/-3/4,3/-5/5,1/-7/6,-1/-7/7,-3/-5/8,-5/-3/9,-5/1/10,-3/3/11,
								-1/7/b_0,1/7/b_1,4.85/4.85/b_2,6.75/2.75/b_3,6.7/-4.7/b_4,4.85/-6.85/b_5,1/-9/b_6,-1/-9/b_7,-4.85/-6.85/b_8,-6.7/-4.7/b_9,-4.85/4.85/b_11,-6.75/2.75/b_10,-3/7/c_0,3/7/c_1,-3/-9/c_4,3/-9/c_3,10/-1/c_2,-10/-1/c_5}
							{
								\node[vert] (\z) at (\x,\y){};
							} 
							
							\foreach \x/\y in {0/1,1/2,2/3,3/4,4/5,5/6,6/7,7/8,8/9,9/10,10/11,11/0}{\path[edge] (\x) -- (\y);}
							\foreach \x/\y in {c_0/c_1,c_1/c_2,c_2/c_3,c_3/c_4,c_4/c_5,c_5/c_0}{\path[edge,dotted] (\x) -- (\y);}
							\foreach \x/\y in {0/1,2/3,4/5,6/7,8/9,10/11}{\draw [line width=2.5pt, line cap=round, dash pattern=on 0pt off 2\pgflinewidth]  (\x) -- (\y);}
							\foreach \x/\y in {0/b_0,1/b_1,2/b_2,3/b_3,4/b_4,5/b_5,6/b_6,7/b_7,8/b_8,9/b_9,10/b_10,11/b_11}{\path[edge,dashed] (\x) -- (\y);}
							
							\draw [-{Stealth[scale=2]}] (-0.2,7) -- (0.2,7); 
							\draw [-{Stealth[scale=2]}] (-0.2,-9)--(0.2,-9);  
							\draw [-{Stealth[scale=2]}] (5.5,4.1)--(5.7,3.9); 
							\draw [-{Stealth[scale=2]}] (6,3.6)--(6.2,3.4);

							\draw [-{Stealth[scale=2]}] (-6.4,-5.15)--(-6,-5.55); 
							\draw [-{Stealth[scale=2]}] (-6,-5.55)--(-5.6,-5.95); 
							
							\draw [-{Stealth[scale=2]}] (-5.5,4.15)--(-5.7,3.9); 
							\draw [-{Stealth[scale=2]}] (-6,3.6)--(-6.2,3.4); 
							\draw [-{Stealth[scale=2]}] (-6.5,3.1)--(-6.7,2.9); 
							
							\draw [-{Stealth[scale=2]}] (6.4,-5.15)--(6,-5.55); 
							\draw [-{Stealth[scale=2]}] (6,-5.55)--(5.6,-5.95); 
							\draw [-{Stealth[scale=2]}] (5.6,-5.95)--(5.2,-6.45); 
							\node[ver] at (0,-11){$(c)$};
						\end{scope}
						
						\begin{scope}[shift={(11,-30)}, scale=1.9]
							\foreach \x/\y/\z in {-2/12/0,2/12/1,-4/10/2,4/10/3,-2/8/4,2/8/5,-4/6/6,4/6/7,-2/4/8,2/4/9,-4/2/10,4/2/11,
								-2/0/12,2/0/13,-4/-2/14,4/-2/15,-2/-4/16,2/-4/17,-4/-6/18,4/-6/19,-2/-8/20,2/-8/21,-4/14/0''',4/14/1'''}
							{
								\node[vert] (\z) at (\x,\y){};
							} 
							
							\foreach \x/\y/\z in {-3/3/a_1,-2.5/3.5/a_2,-3.5/2.5/a_3,
								3/3/a_1,2.5/3.5/a_2,3.5/2.5/a_3}
							{
								\node[vertii] (\z) at (\x,\y){};
							} 
							
							\foreach \x/\y/\z in {-8/11/2',-8/9/2'',-8/7/6',-8/5/6'',-8/3/10',-8/1/10'',
								-8/-1/14',-8/-3/14'',-8/-5/18',-8/-7/18'',
								8/11/3',8/9/3'',8/7/7',8/5/7'',8/3/11',8/1/11'',8/-1/15',8/-3/15'',
								8/-5/19',8/-7/19''}
							{
								\node[vert] (\z) at (\x,\y){};
							}

							\foreach \x/\y/\z in {-8/14/c_1,8/14/c_2,8/-8/c_3,-8/-8/c_4}
							{
								\node[vert] (\z) at (\x,\y){};
							} 
							\foreach \x/\y in {c_1/c_2,c_2/c_3,c_3/c_4,c_4/c_1}{\draw[edge,dotted] (\x) -- (\y);}

							\foreach \x/\y in {0/1,1/3,3/5,5/4,4/2,2/0,5/7,7/9,9/8,8/6,6/4,11/13,13/12,12/10,15/17,17/16,16/14,
								17/19,19/21,21/20,20/18,18/16,1/1''',0/0''',12/14,13/15}{\path[edge] (\x) -- (\y);}
							
							\foreach \x/\y in {2/2',2/2'',3/3',3/3'',6/6',6/6'',7/7',7/7'',
								10/10',10/10'',11/11',11/11'',14/14',14/14'',15/15',15/15'',18/18',18/18'',19/19',19/19''}{\path[edge] (\x) -- (\y);}

							\draw[edge] plot [smooth,tension=1.5] coordinates{(0) (0,10.5) (1)};
							\draw[edge] plot [smooth,tension=1.5] coordinates{(4) (0,6.5) (5)};
							\draw[edge] plot [smooth,tension=1.5] coordinates{(8) (0,2.5) (9)};
							\draw[edge] plot [smooth,tension=1.5] coordinates{(12) (0,-1.5) (13)};
							\draw[edge] plot [smooth,tension=1.5] coordinates{(16) (0,-5.5) (17)};

							\draw [-{Stealth[scale=2]}] (-0.2,14) -- (0.7,14); 
							\draw [-{Stealth[scale=2]}] (-0.2,-8) -- (0.7,-8); 
							
							\draw [-{Stealth[scale=2]}] (-8,4.4) -- (-8,4.6); 
							\draw [-{Stealth[scale=2]}] (-8,3.8) -- (-8,4); 
							
							\draw [-{Stealth[scale=2]}] (8,4.4) -- (8,4.6); 
							\draw [-{Stealth[scale=2]}] (8,3.8) -- (8,4); 
							
							\node[ver](a1) at (0,11.5){{$0,1,\dots,$}};
							\node[ver] (a2) at (0,11){$ d-2$};
							
							\node[ver](a1) at (0,7.5){{$2,3,\dots,d$}};
							
							\node[ver](a1) at (0,3.5){{$4,\dots,d,0,1$}};
							
							\node[ver](a1) at (0,-4.5){{$d-1,d,0,$}};
							\node[ver] (a2) at (0,-5){$\dots,d-4$};
							
							\node[ver](a1) at (0,-0.5){{$d-3,d-2,$}};
							\node[ver] (a2) at (0,-1){$\dots,d-6$};
							
							\node[ver](a1) at (-3.9,13){{$d-1$}};
							\node[ver](a1) at (3.9,13){{$d-1$}};
							\node[ver](a1) at (-3.4,11){{$d$}};
							\node[ver](a1) at (3.4,11){{$d$}};
							\node[ver](a1) at (-3.4,9){{$0$}};
							\node[ver](a1) at (3.4,9){{$0$}};
							\node[ver](a1) at (-3.4,7){{$1$}};
							\node[ver](a1) at (3.4,7){{$1$}};
							\node[ver](a1) at (-3.4,5){{$2$}};
							\node[ver](a1) at (3.4,5){{$2$}};
							
							\node[ver](a1) at (-3.7,0.8){{$d-5$}};
							\node[ver](a1) at (3.7,0.8){{$d-5$}};
							\node[ver](a1) at (-3.7,-0.8){{$d-4$}};
							\node[ver](a1) at (3.7,-0.8){{$d-4$}};
							\node[ver](a1) at (-3.7,-3.2){{$d-3$}};
							\node[ver](a1) at (3.7,-3.2){{$d-3$}};
							\node[ver](a1) at (-3.7,-4.8){{$d-2$}};
							\node[ver](a1) at (3.7,-4.8){{$d-2$}};
							\node[ver](a1) at (-3.9,-7){{$d-1$}};
							\node[ver](a1) at (3.9,-7){{$d-1$}};
							
							\node[ver](a1) at (6.9,10){{$1,2,\dots$}};
							\node[ver](a1) at (-6.8,10){{$\dots,d-1$}};
							\node[ver](a1) at (6.9,6){{$3,4,\dots$}};
							\node[ver](a1) at (-7.2,6){{$\dots,0$}};
							
							\node[ver](a1) at (6.4,2){{$d-4,d-3,$}};
							\node[ver](a1) at (-6.7,2){{$\dots,d-5$}};
							\node[ver](a1) at (6.4,-2){{$d-2,d-1,$}};
							\node[ver](a1) at (-6.7,-2){{$\dots,d-3$}};
							\node[ver](a1) at (6.7,-6){{$d,0,\dots$}};
							\node[ver](a1) at (-6.7,-6){{$\dots,d-3$}};
							\node[ver] at (0,-9){$(d)$};
						\end{scope}
					\end{tikzpicture}

\caption{$(a)$ Embedding on $\mathbb{S}^1\times \mathbb{S}^1$ of gem representing $\mathbb{S}^1\times \mathbb{S}^1$  of type $(6^3)$,  $(b)$Embedding on $\mathbb{S}^1\times \mathbb{S}^1$ of gem representing $\mathbb{S}^1\times \mathbb{S}^1$  of type $(4^1,8^2)$, $(c)$ Embedding on $\mathbb{S}^1\times \mathbb{S}^1$ of gem representing $\mathbb{S}^1\times \mathbb{S}^1$  of type $(4,6,12)$ and $(d)$ Embedding on $\mathbb{S}^1\times \mathbb{S}^1$ of gem representing $\mathbb{S}^{d-1}\times \mathbb{S}^1$  of type $(2^{d-2},6^3)$.}\label{fig:3)}
\end{figure}

\noindent {\bf Acknowledgement:} The first author is supported by the Mathematical Research Impact Centric Support (MATRICS) Research Grant (MTR/2022/000036) by SERB (India). The authors are thankful to the reviewer for the simplified proof of Lemma \ref{lemma:(6^2,4)} and for the comments that improved the presentation of the paper in its current form.
				
\medskip

{\footnotesize


\begin{thebibliography}{99}
						
\bibitem{bk08} U. Brehm and W. Kühnel, Equivelar maps on the torus, {\em European Journal of Combinatorics} {\bf 29}(8) (2008), 1843--1861.						

\bibitem{bm08} J. A. Bondy and U. S. R. Murty, {\em Graph Theory}, Springer, New York, 2008.
																
\bibitem{c89} A. Cavicchioli, A combinatorial characterization of $S^3\times S^1$ among closed
 4–manifolds, {\em Proc. Amer. Math. Soc.} {\bf 105} (4) (1989), 1008--1014.
						
\bibitem{ch93} A. Cavicchioli and F. Hegenbarth, On the determination of PL–manifolds
 by handles of lower dimension, {\em Topology Appl.} {\bf 53} (1993), 111--118.
						
\bibitem{cs93} A. Cavicchioli and F. Spaggiari, On the topological structure of compact
5–manifolds, {\em Comment. Math. Univ. Carolinae} {\bf 34} (3) (1993), 513--524.
						
						
\bibitem{dm17} B. Datta and D. Maity, Semi-equivelar and vertex-transitive maps on the torus, {\em Beitr. Algebra Geom.} {\bf 58} (3) (2017) 617--634.
						
\bibitem{dm18} B. Datta and D. Maity,  Semi-equivelar maps on the torus are Archimedean, {\em Discrete Math.} {\bf 341} (12) (2018) 3296--3309.
						
\bibitem{dm22} B. Datta and D. Maity, Platonic solids, Archimedean solids and semi-equivelar maps on the sphere, {\em Discrete Math.} {\bf 345} (2022), no. 1, Paper No. 112652, 13 pp.
						
\bibitem{fg82} M. Ferri and C. Gagliardi, The only genus zero $n$-manifold is $\mathbb{S}^n$, {\em Proc. Amer. Math. Soc.} {\bf 85} (1982), 638--642.
						
\bibitem{g81} C. Gagliardi, Extending the concept of genus to dimension $n$, {\em Proc. Amer. Math. Soc.} {\bf 81} (1981), 473--481.
						
\bibitem{h76} J. Hempel, 3-Manifolds, Annals of Mathematics Studies 86 (Princeton University Press, Princeton, NJ, 1976).
						
\bibitem{mtu14} D. Maity, A. K. Tiwari and A. K. Upadhyay, Semi-equivelar maps, {\em Beitr. Algebra Geom.} {\bf 55} (2014), 229--242.				
\end{thebibliography}
\end{document}